%% file: main.tex
\documentclass[a4paper,onecolumn,11pt,unpublished]{quantumarticle}
\pdfoutput=1
\input{preamble.tex}
\hypersetup{pdftitle={Local Anticoncentration for Gaussian Boson Sampling via Conditional Wishart Geometry}}
\begin{document}
\title{Local Anticoncentration for Gaussian Boson Sampling via Conditional Wishart Geometry}
\author{Hongru Zhao}
\affiliation{School of Statistics, University of Minnesota, Minneapolis, Minnesota, USA}
\email{zhao1118@umn.edu}
\begin{abstract}
Local anticoncentration controls the conversion of additive output
probability estimates to relative estimates in Gaussian boson sampling.
We prove a center-uniform quadratic small-ball bound for hafnians of
complex Gaussian transpose Gram matrices for $2n$ selected modes, with an explicit finite coefficient
that is polynomial when the number of Gaussian rows is at least of order
$n^2/\log n$. Letting the Gaussian row count tend to infinity yields
anticoncentration for the independent complex symmetric Gaussian ensemble.
For the transpose Gram hafnian bound, the proof compresses dependent hafnian
cofactors by Fourier interpolation and controls their self-coupled variance
through a conditional Wishart identity that preserves the transpose Gram
coordinate. Independent perturbations satisfy related bounds, and a positive
mixture density shows that the local quadratic power is exact. Together
with the uniform hiding theorem in the companion paper, the result gives
local anticoncentration for finite Haar interferometers.
\end{abstract}
\maketitle
\input{sections/00_quantum_introduction}
\input{sections/01_introduction}
\input{sections/02_preliminaries}
\input{sections/03_fourier_compression}
\input{sections/04_conditional_wishart}
\input{sections/05_main_proof}
\input{sections/06_conditioning_consequences}
\input{sections/07_discussion}

\section*{Author contributions and use of artificial intelligence}
Hongru Zhao is responsible for the mathematical content, its presentation,
and the accompanying source materials. OpenAI's GPT-5.6 Sol and GPT-6 Astra Ultra models assisted with
literature searches, manuscript restructuring and editing, algebraic checks,
and preparation of LaTeX and Lean source materials.

\section*{Data and code availability}
No experimental data were created or analyzed. Unconditional Lean verification
of Theorems~\ref{thm:main-finite} and~\ref{thm:symmetric-gaussian-anticoncentration}
is maintained in the GitHub repository~\cite{ZhaoComplexGramHafniansGitHub2026}. The broader formalization and equation-to-code
crosswalk for the archived manuscript version are available on
Zenodo~\cite{ZhaoAnticoncentrationLeanZenodo2026}. The proved endpoints use
no additional scientific axioms.

\clearpage
\appendix
\input{sections/appendix_fourier_algebra}
\input{sections/appendix_wishart_boundary}
\input{sections/appendix_second_moment}
\input{sections/appendix_symmetric_limit}
\input{sections/appendix_angular_obstruction}

\input{sections/appendix_conditioning_details}

\bibliographystyle{quantum}
\bibliography{references}
\end{document}

%% file: sections/00_quantum_introduction.tex
\section{Introduction}
\label{sec:quantum-introduction}
Anticoncentration is one of the analytic requirements in approximate
sampling arguments for quantum computational advantage. An additive
estimate of an output probability gives useful relative accuracy only when
the probability is not too small. Aaronson and Arkhipov formulated this
obstruction for ordinary boson sampling
\cite{AaronsonArkhipov2013}. Gaussian boson sampling (GBS) replaces single
photons by squeezed Gaussian inputs and replaces permanents by hafnians
\cite{LundEtAl2014GaussianState,HamiltonEtAl2017GBS,KruseEtAl2019GBS}.
Its complexity analysis therefore raises the corresponding lower-tail
question for random Gaussian hafnians.

We prove local anticoncentration for both Gaussian ensembles used in that
analysis. If $X\in\C^{k\times2n}$ has independent standard circular
complex Gaussian entries and $H_{k,n}=\haf(X^{\T}X)$, then for $k\ge4n$
\begin{equation}
 \sup_{z\in\C}\Pp\{|H_{k,n}-z|\le\varepsilon\sigma_{k,n}\}
       \le\min\{1,B_{k,n}\varepsilon^2\},
 \label{eq:intro-local-result}
\end{equation}
where $\sigma_{k,n}^2=\E|H_{k,n}|^2$ and the coefficient is explicit.
It is polynomial in $n$ whenever $n^2/k=O(\log n)$.
For the independent complex symmetric Gaussian ensemble, the corresponding
coefficient is at most $2\sqrt{n/\pi}$. The exponent two is the natural
planar disk scaling, uniform over the center and independent of polynomial
degree. A positive continuous mixture density also shows that this local
power is exact in the stated finite-dimensional regime.

Hamilton et al.\ and Kruse et al.\ introduced hafnian anticoncentration as
an ingredient connecting additive squared-hafnian estimation to relative
estimation \cite{HamiltonEtAl2017GBS,KruseEtAl2019GBS}.
Their discussion distinguishes the estimation conjectures from the
Gaussian lower-tail property needed in that connection.
Corollary~\ref{cor:historical-lower-tail} below supplies an explicit
polynomial lower-tail bound for the independent symmetric Gaussian
hafnian, proving the hafnian analogue of Aaronson and Arkhipov's
Permanent Anti-Concentration Conjecture
\cite[Conjecture~1.6, Eq.~(1.3)]{AaronsonArkhipov2013}.
The finite transpose
Gram theorem supplies the analogous property in a specified growth regime.
We do not identify either result with a proof of the broader
average-case hardness or oracle-equivalence assertions. In particular,
control of a squared modulus alone does not recover the phase of a complex
hafnian.

Earlier GBS analyses developed exact moments and located regimes of weak
anticoncentration \cite{EhrenbergEtAl2025Transition,EhrenbergEtAl2025SecondMoment}.
Zhao \cite[Eq.~(8)]{ZhaoExactMoments2026} and, independently, Shou et al.\
\cite[Definition~1 and Theorem~2.1]{ShouEtAl2026AnticoncentrationEntanglement}
establish the $n^2/\log n$ scaling boundary for weak anticoncentration
under the second/fourth-moment criterion.
A moment ratio can ensure a nonnegligible chance of seeing a moderately
large output weight, but it does not by itself bound the probability of an
arbitrarily small disk. This distinction has also been emphasized in the
photonic benchmarking literature \cite{KolarovszkiEtAl2026Framework}.
Our theorem addresses the local requirement directly, in contrast with
generic polynomial small-ball estimates whose radius exponent deteriorates
with the degree \cite{CarberyWright2001}.

The dependent product law is the main mathematical difficulty. Conditioning
on the last column of $X$ gives a circular Gaussian amplitude whose
variance is an energy of hafnian cofactors computed from the same Gaussian
factor. The cofactors cannot be treated as independent of that factor.
We adapt the Fourier compression mechanism of Koehler and Leung for
Ginibre permanents \cite{KoehlerLeung2026}, using a shared-map hafnian
expansion that retains the dependence exactly. A conditional Wishart
identity then changes the Hermitian Gram coordinate while preserving the
transpose Gram coordinate on which the cofactors depend. This controls the
self-coupled inverse variance and yields the coefficient in
Eq.~\eqref{eq:intro-local-result}. The preserved-coordinate inequality is
stated for general measurable directions before its hafnian specialization.

The results apply to ideal collision-free GBS with equal squeezing. They
also include independent perturbations of the Gaussian factor or the
independent symmetric matrix; these are mathematical perturbations, not a
model of arbitrary optical noise. The optical setting is motivated by
photonic implementations \cite{ZhongEtAl2020Photons,MadsenEtAl2022Programmable,DengEtAl2023PseudoPNR},
while classical simulation and realistic error models remain separate
benchmarks \cite{QuesadaArrazola2020ExactSimulation,BulmerEtAl2022Boundary}.

The companion article, \emph{Uniform Hiding and Two Routes to Relative
Accuracy in Gaussian Boson Sampling}~\cite{ZhaoUniformHiding2026}, proves
the finite-Haar replacement and compares the two Gaussian routes at one
absolute additive threshold. Its hiding theorem resolves the quantitative
product-hiding conjecture of Ehrenberg et al.
\cite[Conjecture~1 and Supplemental Eq.~(S62)]{EhrenbergEtAl2025Transition}.
The present Gaussian proof does not use that theorem. Their composition
is recorded in Corollary~\ref{cor:finite-haar-companion}, so the quantum
consequence is visible here without duplicating the companion's
interferometer proof or two-route comparison.

%% file: sections/01_introduction.tex
\section{Physical setting and Gaussian reference}

There are $M$ optical modes, the first $K$ inputs are identical single mode
squeezed vacuum states, and the remaining inputs are vacuum.  A passive Haar
random interferometer mixes the modes before photon number resolving
detection.  A prescribed event without collisions records one photon in each of
$N=2n$ selected output modes.  Selecting those output rows and the squeezed
input columns produces a rectangular interferometer block.  Its product with
its ordinary transpose is the complex symmetric hafnian kernel in the GBS
probability formula \cite{HamiltonEtAl2017GBS,KruseEtAl2019GBS}.  The
transpose is not the adjoint.  Figure~\ref{fig:gbs-setup} summarizes this
physical dictionary.

\begin{figure}[tbp]
  \centering
  \includegraphics[width=0.92\textwidth]{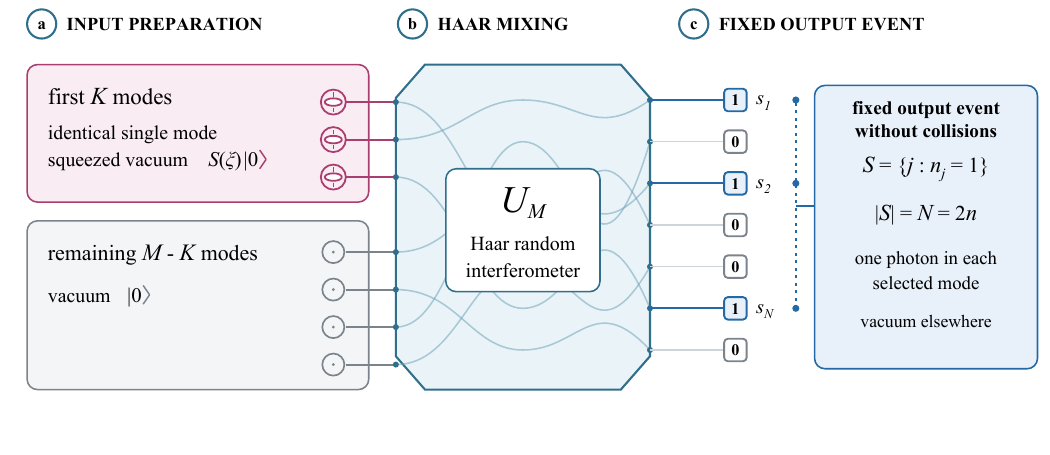}
  \caption{Ideal Gaussian boson sampling and its optical dictionary.
  Identically squeezed vacuum states and vacuum states enter a passive Haar
  random interferometer.  Photon number resolving detectors record a fixed
  output event without collisions.  The selected interferometer rows and
  squeezed input columns form a rectangular block; its transpose-gram product
  is the complex symmetric matrix whose hafnian determines the event
  probability.}
  \label{fig:gbs-setup}
\end{figure}

{The theorem below studies the corresponding Gaussian
transpose-gram product.  We orient the Gaussian matrix as
$X\in\C^{k\times2n}$, so in the GBS correspondence $k=K$ is the number of
squeezed inputs and $2n=N$ is the number of selected output modes.}  This
Article proves the Gaussian theorem and does not use a finite Haar
approximation.

For a complex symmetric $2n\times2n$ matrix $B=(b_{ij})$, define
\begin{equation}
 \haf(B)=\sum_{\mu\in\mathcal M_{2n}}
   \prod_{\{i,j\}\in\mu} b_{ij},
 \label{eq:haf-definition}
\end{equation}
where $\mathcal M_{2n}$ is the set of perfect matchings.  Let
\begin{equation}
 H_{k,n}:=\haf(X^{\T}X),
 \qquad X\in\C^{k\times2n},
 \label{eq:H-definition-intro}
\end{equation}
with independent standard circular complex Gaussian entries.  The transpose
is not a conjugate transpose.  Hence $H_{k,n}$ is a homogeneous holomorphic
Gaussian polynomial of degree $2n$ built from a dependent complex symmetric
matrix.

We call $X^{\T}X$ the transpose-gram matrix.  The distinct matrix $X^*X$
will always be called the Hermitian Gram matrix.

Exact moment analyses for GBS reveal a transition with the number of squeezed
inputs \cite{EhrenbergEtAl2025Transition,EhrenbergEtAl2025SecondMoment,
ZhaoExactMoments2026}.  Our local result instead controls every disk center and
every radius with the natural planar exponent two and an explicit coefficient.

\subsection{Main results}

Set $\sigma_{k,n}^2:=\E|H_{k,n}|^2$ and
\begin{equation}
 b_n:=\frac{2\Gamma(n+\tfrac12)}{\sqrt\pi\,\Gamma(n)}.
 \label{eq:b-limit}
\end{equation}

{The exact hafnian second moment is restated and its two
derivations are compared in
{\cref{prop:exact-second-moment}}:}
{
\begin{equation}
 \sigma_{k,n}^2=(2n-1)!!\prod_{q=0}^{n-1}(k+2q).
 \label{eq:sigma-main}
\end{equation}
}

\begin{theorem}[Uniformly shifted quadratic anticoncentration]
\label{thm:main-finite}
Let $n\ge1$ and $k\ge4n$.  For every $z\in\C$ and $\varepsilon\ge0$,
\begin{equation}
 \Pp\{|H_{k,n}-z|\le\varepsilon\sigma_{k,n}\}
 \le \min\{1,B_{k,n}\varepsilon^2\},
 \label{eq:main-bound}
\end{equation}
where
\begin{align}
 B_{k,n}
 &=b_n\frac{k}{k-1}\prod_{r=2}^{n}\frac{k+2r-2}{k-4r+1},
 \label{eq:B-exact}\\
 B_{k,n}
 &\le b_n\exp\!\left\{
 \frac{3n^2-2}{k}+\frac{9n^3}{k(k-4n+1)}\right\}.
 \label{eq:B-simple}
\end{align}
For fixed $n$, $B_{k,n}\to b_n$ as $k\to\infty$.  If $k\ge8n$, then
\begin{equation}
 B_{k,n}=b_n\exp\!\left(\frac{3n^2}{k}+\theta_{k,n}\right),
 \qquad
 |\theta_{k,n}|\le\frac2k+\frac{94n^3}{k^2}.
 \label{eq:B-bridge}
\end{equation}
\end{theorem}

\begin{proof}
The probabilistic proof is given in \cref{sec:proof-main-finite}.  The
coefficient identities and bounds that complete the argument are proved in
Appendix~\ref{app:second-moment}.
\end{proof}

\begin{corollary}[Independent additive perturbations before the transpose-gram map]
\label[corollary]{cor:independent-matrix-shift}
Let $n\ge1$ and $k\ge4n$.  Let
$\Delta\in\C^{k\times2n}$ be an arbitrary random matrix, independent of the
standard complex Gaussian matrix $X\in\C^{k\times2n}$, with no moment
assumption.  Set

\begin{equation}
 \widetilde H_{k,n}
 :=\haf\!\left((X+\Delta)^{\T}(X+\Delta)\right).
 \label{eq:independent-matrix-shift-observable}
\end{equation}

Then, for every $w\in\C$ and $\varepsilon\ge0$,
\begin{equation}
 \Pp\!\left\{
   \abs{\widetilde H_{k,n}-w}
   \le\varepsilon\sigma_{k,n}
 \right\}
 \le\min\!\left\{1,B_{k,n}\varepsilon^2\right\},
 \label{eq:independent-matrix-shift-bound}
\end{equation}
where $\sigma_{k,n}$ is the centered Gaussian scale in
\cref{thm:main-finite}.
\end{corollary}

\begin{proof}
See Appendix~\ref{app:direct-symmetric}, beginning with the proof of
\cref{cor:independent-matrix-shift}.
\end{proof}

Let $G_{2n}^{\mathrm{sym}}$ denote a complex symmetric $2n\times2n$
Gaussian matrix whose entries on and above the diagonal are independent,
whose entries away from the diagonal are standard circular complex Gaussian, and whose
diagonal entries have complex variance two.
Equivalently, the diagonal normalization is
$\E\abs{(G_{2n}^{\mathrm{sym}})_{ii}}^2=2$.
Put
\begin{equation}
 H_n^{\mathrm{sym}}:=\haf(G_{2n}^{\mathrm{sym}}),
 \qquad
 (\sigma_n^{\mathrm{sym}})^2:=(2n-1)!!.
 \label{eq:symmetric-gaussian-scale}
\end{equation}
This scale is the exact second moment. In the perfect-matching expansion,
each monomial has mean zero, and the cross moment of two distinct matchings
vanishes because an edge appears in only one of them. Each matching has
unit squared moment. There are $(2n-1)!!$ matchings, so
$\E H_n^{\rm sym}=0$ and
$\E|H_n^{\rm sym}|^2=(2n-1)!!$.
The diagonal convention is immaterial. This direct calculation avoids
inferring moment convergence from weak convergence alone.

\begin{theorem}[Complex symmetric Gaussian anticoncentration]
\label{thm:symmetric-gaussian-anticoncentration}
For every $n\ge1$, $z\in\C$, and $\varepsilon\ge0$,
\begin{equation}
 \Pp\!\left\{
   \abs{H_n^{\mathrm{sym}}-z}
   \le \varepsilon\sigma_n^{\mathrm{sym}}
 \right\}
 \le \min\{1,b_n\varepsilon^2\}.
 \label{eq:symmetric-gaussian-small-ball}
\end{equation}
The coefficient $b_n$ is the fixed $n$ limit of the finite coefficient
$B_{k,n}$ in \cref{thm:main-finite}.
\end{theorem}

The coefficient satisfies $b_n\le 2\sqrt{n/\pi}$; see
Appendix~\ref{app:second-moment}.

\begin{proof}
The fixed degree limit proof is given in
Appendix~\ref{app:direct-symmetric}.
\end{proof}

\begin{corollary}[Independent additive perturbations]
\label[corollary]{cor:symmetric-gaussian-perturbation}
Let $\Delta$ be any random complex symmetric $2n\times2n$ matrix independent of
$G_{2n}^{\mathrm{sym}}$.  No boundedness or moment assumption is imposed on
$\Delta$.  Then, for every $z\in\C$ and $\varepsilon\ge0$,
\begin{equation}
 \Pp\!\left\{
  \abs{\haf(G_{2n}^{\mathrm{sym}}+\Delta)-z}
  \le\varepsilon\sigma_n^{\mathrm{sym}}
 \right\}
 \le\min\{1,b_n\varepsilon^2\}.
 \label{eq:symmetric-gaussian-perturbation-bound}
\end{equation}
\end{corollary}

\begin{proof}
See the proof of \cref{cor:symmetric-gaussian-perturbation} in
Appendix~\ref{app:direct-symmetric}.
\end{proof}

The finite perturbation result, \cref{cor:independent-matrix-shift}, acts on
the rectangular factor before the transpose-gram map and is proved in
Appendix~\ref{app:direct-symmetric}.  Its normalized coefficient is the same $B_{k,n}$ as in the
unperturbed theorem. The proof uses direct Gaussian disk centering for each
column. The separate result
\cref{cor:symmetric-gaussian-perturbation} adds $\Delta$ directly to the
complex symmetric matrix and has the sharper coefficient $b_n$.

The symmetric Gaussian theorem is obtained by sending the row dimension to
infinity in \cref{thm:main-finite}.  Its proof uses the finite dimensional
matrix central limit theorem, continuity and homogeneity of the hafnian, and
a normalized Portmanteau argument.  The perturbation corollary is proved
directly from the independent off diagonal Gaussian entries.  Diagonal
entries never appear in a hafnian monomial, and successive centering of the
Gaussian edge stars gives the perturbation bound without a boundedness or
moment assumption.  The full arguments are given in
Appendix~\ref{app:direct-symmetric}.

The radial and angular decomposition in \cref{lem:radial-angular-factorization}
also exposes a distinct negative result: \cref{thm:wick-angular-obstruction}
proves an exponential Wick lower bound for the angular factor when
$n\ge1000$ and $2\le k\le n/3$.

The exact product form makes the useful parameter window in finite dimensions
transparent.
{The normalized penalty $B_{k,n}/b_n$ in finite dimensions is
polynomial when $n^2/k=O(\log n)$, remains of constant order when
$k\asymp n^2$, and tends to one when $k$ is superquadratic.  At the edge
$k=4n$, $B_{k,n}/b_n$ is exponential, so the condition $k\ge4n$ alone is not
a claim of an efficient complexity regime.}

Section~\ref{sec:conditioning-consequences} derives additive-to-relative
probability transfer and negative fractional moments. These use the local
lower tail directly. The construction and running time of an additive
estimator remain separate assumptions in any complexity reduction.

\begin{corollary}[Polynomial lower tails for the Gaussian hafnian]
\label{cor:historical-lower-tail}
For every $n\ge1$, $0<\delta<1$, and $z\in\C$,
\begin{equation}
 \Pp\!\left\{ |H_n^{\rm sym}-z|
    \le\frac{\delta}{2n}\sqrt{(2n-1)!!}\right\}\le\delta.
 \label{eq:historical-polynomial-bound}
\end{equation}
Thus the polynomial $q(n,1/\delta)=2n/\delta$ controls the inverse
relative amplitude scale in the Gaussian lower-tail requirement.
\end{corollary}
\begin{proof}
Apply Theorem~\ref{thm:symmetric-gaussian-anticoncentration} with
$\varepsilon=\delta/(2n)$. Since $b_n\le2\sqrt{n/\pi}\le2n$,
$b_n\varepsilon^2\le\delta^2/(2n)\le\delta$.
\end{proof}

At $z=0$, Eq.~\eqref{eq:historical-polynomial-bound} is the analytic
Gaussian hafnian counterpart of Aaronson and Arkhipov's Permanent
Anti-Concentration Conjecture
\cite[Conjecture~1.6, Eq.~(1.3)]{AaronsonArkhipov2013}: the amplitude
exceeds its root mean square divided by a polynomial with probability
greater than $1-\delta$. Here $\sqrt{n!}$ is replaced by
$\sqrt{(2n-1)!!}$, and $q(n,1/\delta)=2n/\delta$ works; the proof gives
the strict bound $\delta^2/(2n)<\delta$.
This proves the Gaussian lower-tail property anticipated in the GBS
proposals \cite[``Approximate GBS'']{HamiltonEtAl2017GBS}
and \cite[Secs.~V.A--V.B]{KruseEtAl2019GBS}.
Their broader conjecture relating additive squared-hafnian estimation
to multiplicative complex-hafnian estimation also requires a computational
reduction; phase recovery and average-case hardness are not established
by this probability bound.

\begin{corollary}[Polynomial small balls]
\label[corollary]{cor:polynomial}
Fix $A,D>0$.  Suppose an integer sequence $k_n$ is eventually positive and
\begin{equation}
 \frac{n^2}{k_n}\le D\log n
 \quad\text{for all sufficiently large }n.
 \label{eq:poly-regime}
\end{equation}
For every fixed
\begin{equation}
 \gamma>\frac{A+3D+1/2}{2},
 \label{eq:radius-exponent-choice}
\end{equation}
one has, for all sufficiently large $n$ and every $z\in\C$,
\begin{equation}
 \Pp\{|H_{k_n,n}-z|\le n^{-\gamma}\sigma_{k_n,n}\}
 \le n^{-A}.
 \label{eq:poly-cor}
\end{equation}
\end{corollary}

\begin{proof}
See Appendix~\ref{app:second-moment}, Sec.~\ref{sec:polynomial-regime}.
\end{proof}

\begin{corollary}[Finite-Haar consequence with the companion theorem]
\label{cor:finite-haar-companion}
Let $U$ be Haar in $\mathrm U(M)$, $N=2n$, $4n\le K\le M$, and let
$S$ be a fixed set of $N$ output modes. Then
\begin{align}
 &\Pp_U\!\left\{\left|\haf\!\left(MU_{S,[K]}U_{S,[K]}^{\T}\right)-z\right|
                  \le\varepsilon\sigma_{K,n}\right\}
 \nonumber\\
 &\qquad\le\min\{1,B_{K,n}\varepsilon^2+C_\ast N^2/M\},
 \label{eq:finite-haar-companion}
\end{align}
for every $z\in\C$ and $\varepsilon\ge0$, where $C_\ast=615172$.
\end{corollary}
\begin{proof}
The companion's uniform product-hiding theorem compares
$MU_{S,[K]}U_{S,[K]}^{\T}$ with $GG^{\T}$ for
$G\in\C^{N\times K}$ at total variation distance at most
$\min\{1,C_\ast N^2/M\}$ \cite{ZhaoUniformHiding2026}.
Apply it to the inverse image of the disk under the hafnian, and then use
Theorem~\ref{thm:main-finite} with $X=G^{\T}$.
\end{proof}

For $z=0$, this is a lower-tail bound for the ideal finite-interferometer
GBS probability after inserting the usual squeezing prefactor.
The companion derives the common-threshold comparison with the
independent symmetric route, including the different probability
normalizations. This application uses hiding; the proofs of our Gaussian
main theorems do not.

\subsection{Proof strategy}

After conditioning on the last Gaussian column, the proof alternates
Fourier cofactor compression with a preserved-coordinate Wishart inverse
bound. The conditional Gaussian mixture then yields the positive density
and exact local power.

Writing $W_r$ and $V_r$ for the cofactor energies defined in
\cref{sec:preliminaries}, the two inverse moment steps are

\begin{align}
 \mathcal W_{k,r}&\le \frac{1}{2r-2}\,\mathcal V_{k,r-1},
 \label{eq:fourier-overview}\\
 \mathcal V_{k,r}&\le \frac{1}{k-4r+1}\,\mathcal W_{k,r}.
 \label{eq:wishart-overview}
\end{align}

The first is the Fourier compression estimate and the second is the
preserved coordinate Wishart estimate; their hypotheses and endpoint cases
are stated where they are proved.

%% file: sections/02_preliminaries.tex
\section{Gaussian hafnian anticoncentration}
\label{sec:preliminaries}

\subsection{Last column Gaussian mixture and exact scale}

A standard circular complex Gaussian $Z\sim\CN(0,1)$ has density
$\pi^{-1}e^{-\abs z^2}$ on $\C$; equivalently, its real and imaginary parts
are independent $N(0,1/2)$ variables.  More generally,
$Z\sim\CN(0,v)$ means
\[
  p_Z(z)=\frac1{\pi v}e^{-\abs z^2/v},
  \qquad \E\abs Z^2=v.
\]
This is the proper complex Gaussian normalization of
Ref.~\cite{Goodman1963}.
We write $G_a\sim\GammaLaw(a,1)$ for a Gamma random variable with shape $a$
and scale one.  Gamma factors displayed in a product are independent unless
stated otherwise.

We set $\haf(\varnothing)=1$.  For an odd number of Gaussian columns, let
\begin{equation}
  A_r=[x_1,\ldots,x_{2r-1}]
  \in\C^{k\times(2r-1)},
  \label{eq:Ar-definition}
\end{equation}
and define
\begin{equation}
  (C_r)_j
  :=\haf(A_{r,-j}^{\T}A_{r,-j}),
  \qquad 1\le j\le 2r-1,
  \label{eq:Cr-definition}
\end{equation}
where $A_{r,-j}$ is obtained by deleting column $j$.  We use the two energies
\begin{equation}
  W_r:=\norm{C_r}^2,
  \qquad
  V_r:=\norm{A_rC_r}^2.
  \label{eq:Wr-Vr-definition}
\end{equation}
At level $r=1$, this convention gives $C_1=(1)$, $W_1=1$, and
$V_1=\norm{x_1}^2$.

We interpret every inverse moment below as an extended nonnegative expectation.
If
$\mu_{k,r}$ is the product circular Gaussian law of the $2r-1$ past columns,
set
\begin{equation}
 \begin{aligned}
 \mathcal V_{k,r}&:=\int V_r(A)^{-1}\,\dd\mu_{k,r}(A),\\
 \mathcal W_{k,r}&:=\int W_r(A)^{-1}\,\dd\mu_{k,r}(A),\\
 \mathcal V_{k,r},\mathcal W_{k,r}&\in[0,\infty].
 \end{aligned}
 \label{eq:inverse-moment-definitions}
\end{equation}
The quantities $\mathcal V_{k,r}$ and $\mathcal W_{k,r}$ retain this extended
meaning until a finite upper bound has been proved.  We then use the ordinary
expectation notation $\E V_r^{-1}$ or $\E W_r^{-1}$ for the same value.

{For $r\ge2$, the vector $C_r$ is a polynomial function of the
transpose-gram matrix $S_r=A_r^{\T}A_r$, so each of its coordinates is a
holomorphic polynomial in the entries of $A_r$.  To verify that $(C_r)_1$ is
not identically zero for any $k\ge1$, assign every relevant column the same
real unit vector.  Every entry away from the diagonal of the resulting even
Gram submatrix equals one, and hence $(C_r)_1=(2r-3)!!\ne0$.  For $r=1$, we have
already set $C_1=(1)$.  Now identify the complex entries of $A_r$ with their
underlying real coordinates and define
$f(A_r):=\abs{(C_r)_1(A_r)}^2$.  This is a nonzero real valued real analytic
function, and its zero set is exactly the zero set of $(C_r)_1$.  By
Ref.~\cite{Mityagin2020ZeroSet}, that set is Lebesgue null.  The standard
complex Gaussian law is absolutely continuous with respect to Lebesgue
measure, so the same set is Gaussian null.  Consequently, $C_r\ne0$ almost
surely.  Whenever $k\ge2r-1$, a fixed
maximal minor of $A_r$ is a nonzero polynomial, so $A_r$ has full column rank
almost surely.  In particular, under the main theorem hypothesis $k\ge4n$,
this holds at every level $r\le n$ and, since $C_r\ne0$ almost surely, gives
$V_r=\norm{A_rC_r}^2>0$ almost surely.}

\subsubsection{Expansion along the last column}

Conditioning a multilinear Gaussian polynomial on its last block is standard.
Here the paper specific difficulty is that the mixing variance
$V_r=\norm{A_rC_r}^2$ is coupled to a cofactor vector measurable in the same
transpose-gram data.  The following reduction exposes a scalar
complex Gaussian while retaining that coupling.

\begin{proposition}[Fixed past last column law]
\label[proposition]{prop:conditional-mixture}
For $r\ge1$, let $x_{2r}$ be an independent standard complex Gaussian column
and put
\[
  H_{k,r}=\haf([A_r\;x_{2r}]^{\T}[A_r\;x_{2r}]).
\]
For every fixed past matrix $A_r$, put $y(A_r):=A_rC_r$,
$T_{A_r}(x):=x^{\T}y(A_r)$, and
$S_{A_r}(z):=\sqrt{V_r(A_r)}\,z$.  Then
\begin{equation}
 \begin{aligned}
 H_{k,r}(A_r,x)&=T_{A_r}(x),\\
 (T_{A_r})_{\#}\CN(0,I_k)&=(S_{A_r})_{\#}\CN(0,1).
 \end{aligned}
 \label{eq:last-column-map-law}
\end{equation}
{Equivalently, conditional on the fixed past matrix $A_r$,
\begin{equation}
 H_{k,r}\mid A_r\sim\CN\!\left(0,V_r(A_r)\right).
 \label{eq:conditional-law-given-past}
\end{equation}
}
\end{proposition}

\begin{proof}
Expanding the hafnian along the vertex corresponding to $x_{2r}$ gives
\[
  H_{k,r}
  =\sum_{j=1}^{2r-1}(x_{2r}^{\T}x_j)(C_r)_j
  =x_{2r}^{\T}A_rC_r.
\]
For fixed $y\in\C^k$, the random variable $x_{2r}^{\T}y$ is proper complex
Gaussian with variance $\norm y^2$; the transpose rather than adjoint causes
no change in that variance.  {Taking $y=A_rC_r$ proves
\eqref{eq:last-column-map-law}, whose {conditional form on each
fixed fiber} is
\eqref{eq:conditional-law-given-past}.}
\end{proof}

The fixed past law turns a shifted planar small ball into an extended inverse
moment.  {In the range $k\ge2r-1$,} conditional on $A_r$, the
density of $H_{k,r}$ is bounded by
$(\pi V_r)^{-1}$.  A disk of radius $\rho$ has area $\pi\rho^2$, so, for any
center $z_0\in\C$ and radius $\rho\ge0$, almost sure positivity of $V_r$ gives
\begin{equation}
  \Pp\{\abs{H_{k,r}-z_0}\le\rho\}
  \le \rho^2\mathcal V_{k,r}
  \qquad\text{in }[0,\infty].
  \label{eq:disk-to-inverse-moment}
\end{equation}
after averaging over the past matrix.  Thus the proof of
\cref{thm:main-finite} reduces to controlling
$\mathcal V_{k,n}$ with its dependence on $n$ and $k$ explicit.

\subsubsection{The exact second moment}

{The exact hafnian second moment used for the normalization in
Eq.~\eqref{eq:sigma-main} was first proved by Ehrenberg et al.  They expand
the hafnian, represent the surviving Gaussian contractions by connected
components of two superimposed perfect matchings, and evaluate the resulting
weighted graph sum by induction
{\cite[Theorem~1 and Supplemental
Sec.~S3]{EhrenbergEtAl2025Transition}}.  Their later second moment paper
reviews this graph calculation
{\cite[Sec.~III.A]{EhrenbergEtAl2025SecondMoment}}.
Zhao later gave a shorter independent proof that compresses the matching
enumeration into a single auxiliary Gaussian moment
{\cite[Eq.~(3) and Supplemental
Eq.~(S9)]{ZhaoExactMoments2026}}.}

\begin{proposition}[Exact scale]
\label[proposition]{prop:exact-second-moment}
For all $k,n\ge1$,
\begin{equation}
  \E\abs{H_{k,n}}^2
  =(2n-1)!!\prod_{q=0}^{n-1}(k+2q).
  \label{eq:exact-second-moment}
\end{equation}
\end{proposition}

\begin{proof}
See the two cited derivations immediately preceding the proposition:
Ref.~\cite[Theorem~1 and Supplemental Sec.~S3]{EhrenbergEtAl2025Transition}
and Ref.~\cite[Eq.~(3) and Supplemental Eq.~(S9)]{ZhaoExactMoments2026}.
\end{proof}

%% file: sections/03_fourier_compression.tex
\section{Fourier compression for odd hafnian cofactors}
\label{sec:fourier}

This section proves the first inverse moment step in
\eqref{eq:fourier-overview}.  Its analytic mechanism is adapted from the
quadratic scaling in one direction, Gaussian geometric interpolation, and
two coordinate compression of Koehler and Leung
\cite[Lemmas~3.1 and 3.2 and Corollary~3.3]{KoehlerLeung2026}.  Their
cofactor level iteration and Fourier to Laplace passage appear as
\cite[Proposition~3.5 and Lemma~3.4]{KoehlerLeung2026}.  We give a
self contained derivation because the dependent transpose-gram hafnian law requires a
new bilinear cofactor expansion and an exact identification of the
conditional real phase.

The argument has three steps.  A bilinear Gaussian identity isolates the
phase created by two selected coordinates.  Iterating the resulting geometric
mean inequality compresses every test direction to one coordinate.  A positive
Fourier kernel then converts the characteristic function comparison into the
inverse moment estimate required by the recursion.

For $r\ge2$, define the characteristic function of the cofactor vector by
\begin{equation}
  \phi_r(z):=\E\exp\!\left(i\operatorname{Re}(z^*C_r)\right),
  \qquad z\in\C^{2r-1}.
  \label{eq:phi-definition}
\end{equation}

Equivalently, a coefficient $w_j$ in
$\operatorname{Re}\sum_j w_j(C_r)_j$ corresponds to
$z_j=\overline{w_j}$ under the convention in
Eq.~\eqref{eq:phi-definition}.

Figure~\ref{fig:fourier-cofactor-compression} gives an
educational map of the proof from the shared cofactor expansion to the first
inverse moment recursion.

The bilinear interpolation and the complex to real coordinate conversion are
proved in Appendix~\ref{app:fourier-algebra}.  The main text keeps the
Gram specific compression statement and its probabilistic consequence.

Here is the Gram specific algebra that makes the compression work.  Condition
on all columns except two, call the selected columns $X,Y$, and use the
quantities $q_R$ and $M_{j,R}$ defined in Appendix~\ref{app:fourier-algebra}.
Direct hafnian expansion gives
\[
 \begin{aligned}
 (C_r)_1&=Y^{\T}q_R,&
 (C_r)_2&=X^{\T}q_R,\\
 (C_r)_j&=X^{\T}M_{j,R}Y &&(j\ge3).
 \end{aligned}
\]
The first two cofactors therefore share the \emph{same} random linear map;
the remaining ones are bilinear in the same pair.  After realification, these
identities put the full conditional phase into one Gaussian bilinear form.
That exact shared map structure, rather than an independence approximation,
is the new input that permits the two coordinate interpolation below.

\begin{proposition}[Cofactor Fourier compression]
\label[proposition]{prop:cofactor-compression}
The following is the Gram hafnian counterpart of the two coordinate
compression in \cite[Proposition~3.5]{KoehlerLeung2026}.
For every $k\ge1$, $r\ge2$, and $z\in\C^{2r-1}$,
\begin{equation}
  \abs{\phi_r(z)}
  \le \phi_r(\norm z e_1)
  =\E\exp\!\left(-\frac{\norm z^2V_{r-1}}4\right).
  \label{eq:cofactor-compression}
\end{equation}
\end{proposition}

\begin{proof}
See Appendix~\ref{app:fourier-algebra}.
\end{proof}

We next turn this characteristic function bound directly into the extended
inverse moment estimate needed for the induction, following the positive kernel
principle in \cite[Lemma~3.4]{KoehlerLeung2026}.

\begin{proposition}[Fourier inverse moment step]
\label[proposition]{prop:fourier-to-laplace}
For $r\ge2$ and $k\ge2r-1$, the extended inverse moments satisfy
\begin{equation}
  \mathcal W_{k,r}
  \le \mathcal V_{k,r-1}\frac1{2r-2}
  \qquad\text{in }[0,\infty].
  \label{eq:W-inverse-recursion}
\end{equation}
\end{proposition}

\begin{proof}
Realify $\C^{2r-1}$ as $\R^{4r-2}$.  For $t>0$ and $u$ in this real space,
Gaussian Fourier inversion \cite{Folland1992FourierAnalysis} gives
\begin{equation}
  e^{-t\norm u^2}
  =(4\pi t)^{-(2r-1)}
  \int_{\R^{4r-2}}
  e^{-\norm\xi^2/(4t)}e^{i\xi\cdot u}\,\dd\xi.
\end{equation}
The kernel is nonnegative.  Let $\gamma\sim\CN(0,I_{2r-1})$ be independent of
$V_{r-1}$ and set $Y=\sqrt{V_{r-1}}\gamma$.  Its characteristic function is
\begin{equation}
  \phi_Y(z)=\E e^{-\norm z^2V_{r-1}/4}.
\end{equation}
By \cref{prop:cofactor-compression},
$\operatorname{Re}\phi_r(z)\le\abs{\phi_r(z)}\le\phi_Y(z)$.
{Under this inversion, the Gaussian kernel integrals of
$\operatorname{Re}\phi_r$ and $\phi_Y$ are respectively
$\E e^{-t\norm{C_r}^2}$ and $\E e^{-t\norm{Y}^2}$.}
Integrating the inequality against the displayed nonnegative Fourier kernel
yields
\begin{equation}
  \E e^{-tW_r}
  \le\E e^{-tV_{r-1}\norm\gamma^2}.
\end{equation}
Since $G_{2r-1}:=\norm\gamma^2\sim\GammaLaw(2r-1,1)$, its inverse moment is
$\E G_{2r-1}^{-1}=(2r-2)^{-1}$.

Finally, Tonelli's theorem and the Laplace representation
$x^{-1}=\int_0^\infty e^{-tx}\,\dd t$
\cite[Chapters~1 and~2]{Kallenberg2021} give
\begin{align}
  \mathcal W_{k,r}
  &\le\mathcal V_{k,r-1}\,\E G_{2r-1}^{-1}
  =\frac1{2r-2}\mathcal V_{k,r-1},
\end{align}
which proves \eqref{eq:W-inverse-recursion}.  Monotone convergence justifies
the statement without first assuming either inverse moment finite.
{The condition $k\ge2r-1$ ensures $W_r>0$ almost surely and also
$V_{r-1}>0$ almost surely, since $2r-3\le k$; these are precisely the
positivity conditions used above.}
\end{proof}

%% file: sections/04_conditional_wishart.tex
\section{Conditional Wishart geometry}
\label{sec:wishart}

This section isolates the geometric mechanism behind
Eq.~\eqref{eq:V-to-W-inverse}.  The positive definite cone underlying the real
Wishart distribution~\cite{Wishart1928,Muirhead1982} carries two complementary
complex Gram coordinates.  Fixing the transpose-gram coordinate defines a
relatively open fiber, and motion along that fiber varies the Hermitian Gram coordinate
without changing any cofactor data determined by the transpose-gram matrix.
We combine a score identity along these fiber directions with Schur complement
order, first for a general measurable direction and then for the hafnian
cofactor vector.  The word conditional refers to fixing the transpose-gram
coordinate; the proof does not require a formula for a conditional Wishart
density.

At recursion level $r$, take
\[
 \begin{gathered}
 m=2r-1,\qquad c(S)=C_r,\\
 S=A^{\T}A,\qquad Q=A^*A.
 \end{gathered}
\]
Thus $S$ is the preserved coordinate, $Q$ is the varying Hermitian coordinate,
and the cofactor vector is fixed along each $S$ fiber.  We first prove the
$m$ dimensional statement for a general measurable $c(S)$ and then make this
substitution.

\subsection{The Wishart cone and its transpose-gram fibers}

Let $\mathcal P_{2m}:=\{M=M^{\T}\in\R^{2m\times2m}:M\succ0\}$.  Write a point
of this positive definite real Wishart cone in blocks as
\begin{equation}
 \mathcal P_{2m}\ni M=
 \begin{pmatrix}P&C\\C^{\T}&D\end{pmatrix},
 \qquad M\succ0.
 \label{eq:wishart-cone-blocks}
\end{equation}
The linear maps
\begin{equation}
 \begin{aligned}
  \mathfrak s(M)&:=P-D+i(C+C^{\T}),\\
  \mathfrak q(M)&:=P+D+i(C-C^{\T})
 \end{aligned}
 \label{eq:wishart-gram-coordinate-maps}
\end{equation}
take values in the complex symmetric and Hermitian matrices, respectively.
Indeed, for $A=X+iY$, $R=[X\;Y]$, and $M(R)=R^{\T}R$,
\begin{equation}
 \begin{aligned}
 S(R)&:=A^{\T}A=\mathfrak s(M(R)),\\
 Q(R)&:=A^*A=\mathfrak q(M(R)).
 \end{aligned}
 \label{eq:two-gram-coordinates}
\end{equation}
The pair $(\mathfrak s,\mathfrak q)$ is a real linear coordinate system,
because
\begin{equation}
 \begin{aligned}
 P&=\frac{\operatorname{Re}\mathfrak q+\operatorname{Re}\mathfrak s}{2},
 &D&=\frac{\operatorname{Re}\mathfrak q-\operatorname{Re}\mathfrak s}{2},\\
 C&=\frac{\operatorname{Im}\mathfrak s+\operatorname{Im}\mathfrak q}{2}.
 \end{aligned}
 \label{eq:inverse-gram-coordinate-map}
\end{equation}
For a fixed complex symmetric matrix $s$, define the relatively open fiber
\begin{equation}
 \mathcal F_s:=\{M\in\mathcal P_{2m}:\mathfrak s(M)=s\}.
 \label{eq:transpose-gram-fiber}
\end{equation}

\begin{proposition}[Affine kernel coordinates of a transpose-gram fiber]
\label[proposition]{prop:transpose-gram-fiber-tangent}
For a Hermitian \(H=H_0+iH_1\), define the real symmetric direction
\begin{equation}
 \Delta(H):=\frac12
 \begin{pmatrix}H_0&H_1\\-H_1&H_0\end{pmatrix}.
 \label{eq:delta-direction}
\end{equation}
Then
\begin{equation}
 \partial_{\Delta(H)}\mathfrak s=0,
 \quad
 \partial_{\Delta(H)}\mathfrak q=H,
 \quad
 \tr(\Delta(H))=\tr(H).
 \label{eq:delta-coordinate-derivatives}
\end{equation}
Moreover,
\begin{equation}
 \begin{aligned}
 \mathcal K_{\mathfrak s}
 &:=\{D_0=D_0^{\T}:\partial_{D_0}\mathfrak s=0\},\\
 \mathcal K_{\mathfrak s}
 &=\{\Delta(H):H=H^*\}.
 \end{aligned}
 \label{eq:fiber-tangent-space}
\end{equation}
\end{proposition}

\begin{proof}
Write a real symmetric direction as
\(
 D_0=\begin{psmallmatrix}U&V\\V^{\T}&W\end{psmallmatrix}
\).
Since the coordinate maps in Eq.~\eqref{eq:wishart-gram-coordinate-maps} are
linear,
\[
 \partial_{D_0}\mathfrak s=U-W+i(V+V^{\T}).
\]
This vanishes exactly when $W=U$ and $V^{\T}=-V$.  In that case, put
$H_0=2U$ and $H_1=2V$.  Then $H=H_0+iH_1$ is Hermitian and
$D_0=\Delta(H)$.  Direct substitution gives
$\partial_{\Delta(H)}\mathfrak q=H$ and
$\tr(\Delta(H))=\tr(H)$.  Conversely, every Hermitian $H$ gives a real
symmetric $\Delta(H)$ in this kernel, proving the exact affine kernel identity
in Eq.~\eqref{eq:fiber-tangent-space}.
\end{proof}

The proposition identifies the affine fiber directions used below, without
requiring any manifold tangent space formalism.  The maps $H\mapsto
\Delta(H)$ and $D_0\mapsto\partial_{D_0}\mathfrak q$ give inverse coordinates
on the symmetric kernel of the preserved coordinate.  Thus a vector $c(S)$ determined by the
transpose-gram matrix remains fixed along every fiber direction, while its quadratic
form $c(S)^*Q c(S)$ can be varied in an arbitrary Hermitian direction.
Figure~\ref{fig:conditional-wishart-geometry} summarizes this separation.

\begin{figure}[tbp]
 \centering
 \includegraphics[width=0.96\textwidth]{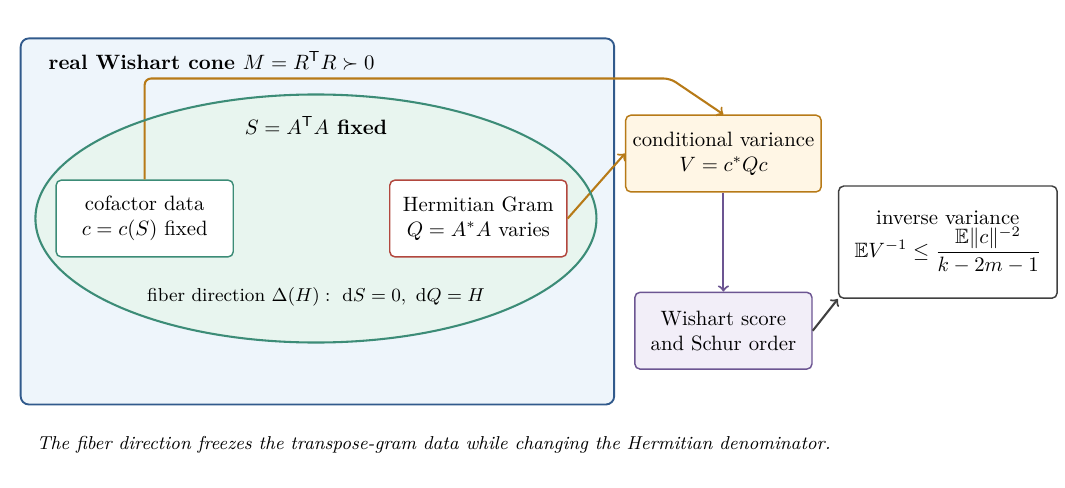}
 \caption{Conditional Wishart geometry.  The real Gram matrix
 $M=R^{\T}R$ moves inside a relatively open fiber on which the transpose-gram
 coordinate $S=A^{\T}A$ is fixed.  The affine fiber direction $\Delta(H)$ therefore
 freezes every vector $c(S)$ while changing the Hermitian Gram coordinate by
 $\partial_{\Delta(H)}Q=H$.  A score identity along these directions and
 Schur complement order control the inverse conditional variance.}
 \label{fig:conditional-wishart-geometry}
\end{figure}

\subsection{A score identity along the fibers}

Let \(A=X+iY\in\C^{k\times m}\), where the entries of \(X,Y\) are independent
\(N(0,1/2)\) variables, and put
\[
 R=[X\;Y]\in\R^{k\times2m},\qquad M(R)=R^{\T}R.
\]
Write \(\mu_{k,m}\) for this product half Gaussian matrix law and set
\begin{equation}
 \alpha_{k,m}:=\frac{k-2m-1}{2}.
 \label{eq:wishart-score-alpha}
\end{equation}
Haff's Theorem~2.1 supplies the classical ambient Wishart integration by parts
framework~\cite{Haff1979}.  The preserved $S$, bounded Borel specialization
below is proved in Appendix~\ref{app:wishart-boundary}
by an unsliced Gaussian coordinate cutoff argument.  Its bounded Borel form is
what permits the direction to depend on the preserved coordinate.

\begin{proposition}[Bounded Borel preserved \(S\) score]
\label[proposition]{prop:bounded-preserved-S-score}
Assume \(k>2m+1\).  Let
\(u:\{s\in\C^{m\times m}:s^{\T}=s\}\to\R\) be Borel measurable and
suppose that \(\abs{u(s)}\le C\) for every symmetric \(s\).  For every Hermitian
\(H\), both integrands below are integrable and
\begin{multline}
 \int u(S(R))\,\alpha_{k,m}
   \tr\!\bigl(M(R)^{-1}\Delta(H)\bigr)\,\dd\mu_{k,m}(R)
 \\
 =\tr\!\bigl(\Delta(H)\bigr)
   \int u(S(R))\,\dd\mu_{k,m}(R).
 \label{eq:bounded-preserved-S-score}
\end{multline}
\end{proposition}

\begin{proof}
See Appendix~\ref{app:wishart-boundary}.
\end{proof}

Because \eqref{eq:bounded-preserved-S-score} holds for every bounded Borel
test of \(S\), it also applies coordinate by coordinate to a bounded
Hermitian matrix valued function of \(S\).  This extension in finite dimensions
will be used below for a regularized direction of rank one.

\subsection{Deterministic Schur geometry}

For a deterministic complex matrix \(A\), set
\[
 Q=A^*A,\qquad S=A^{\T}A,\qquad
 K(A)=\begin{pmatrix}Q&\overline S\\S&\overline Q\end{pmatrix},
\]
and define
\[
 \mathcal S(A):=Q-\overline S\,\overline Q^{-1}S.
\]

\begin{proposition}[Coupled Schur order]
\label[proposition]{prop:deterministic-schur-order}
If the column map of \([A\;\overline A]\) is injective, then
\begin{equation}
 \begin{aligned}
  0&\prec\mathcal S(A)\preceq Q,\\
  Q^{-1}&\preceq\mathcal S(A)^{-1},\\
  (K(A)^{-1})_{11}&=\mathcal S(A)^{-1}.
 \end{aligned}
 \label{eq:deterministic-schur-order}
\end{equation}
Moreover, for every \(c\in\C^m\),
\[
 (c^*Qc)(c^*Q^{-1}c)\ge (c^*c)^2.
\]
\end{proposition}

\begin{proof}
The matrix \(K(A)\) is the Hermitian Gram matrix of
\([A\;\overline A]\), hence it is positive definite under the injectivity
hypothesis.  Its principal block \(Q\) is therefore positive definite, and
the standard Schur complement formula
~\cite[Sec.~7.7]{HornJohnson2013MatrixAnalysis} gives
\[
 \mathcal S(A)=Q-\overline S\,\overline Q^{-1}S\succ0,
 \qquad (K(A)^{-1})_{11}=\mathcal S(A)^{-1}.
\]
The subtracted term is positive semidefinite, so
\(\mathcal S(A)\preceq Q\).  Inversion reverses the positive definite order
and gives \(Q^{-1}\preceq\mathcal S(A)^{-1}\).

For the last assertion, apply ordinary Cauchy Schwarz to
\(Q^{1/2}c\) and \(Q^{-1/2}c\):
\[
 (c^*c)^2
 =\abs{(Q^{1/2}c)^*(Q^{-1/2}c)}^2
 \le(c^*Qc)(c^*Q^{-1}c).
\]
\end{proof}

\subsection{A preserved coordinate inverse variance theorem}

Wishart integration by parts and complex covariance and relation coordinates
are classical~\cite{Haff1979,Goodman1963}.  We have not found in that literature
a score identity restricted to fixed transpose-gram fibers that permits the
direction to be an arbitrary measurable function of the preserved coordinate.
The following theorem gives the resulting inequality before any hafnian
structure is imposed.

\begin{theorem}[Preserved coordinate inverse variance]
\label{thm:preserved-coordinate-inverse-variance}
Let \(m\ge1\) and \(k>2m+1\).  Let \(A\in\C^{k\times m}\) have independent
standard circular complex Gaussian entries, and put
\[
 S=A^{\T}A,\qquad Q=A^*A.
\]
Let
\[
 c:\{s\in\C^{m\times m}:s^{\T}=s\}\longrightarrow\C^m
\]
be Borel measurable and suppose that \(\Pp\{c(S)=0\}=0\).  Define
\[
 W=c(S)^*c(S),\qquad V=c(S)^*Q c(S).
\]
Then, as an inequality of extended nonnegative expectations,
\begin{equation}
 \E V^{-1}
 \le\frac1{k-2m-1}\E W^{-1}.
 \label{eq:preserved-coordinate-inverse-variance}
\end{equation}
\end{theorem}

\begin{proof}
The real matrix \(R=[\operatorname{Re}A\;\operatorname{Im}A]\) has \(2m\)
Gaussian columns and full column rank almost surely.  The complex matrix
\([A\;\overline A]\) is obtained from \(R\) by a fixed invertible change of
coordinates, so \cref{prop:deterministic-schur-order} applies almost surely.
The hypothesis on \(c\) and positive definiteness of \(Q\) give \(W,V>0\)
almost surely.

Abbreviate \(c=c(S)\) and define
\[
 \Xi:=c^*\mathcal S(A)^{-1}c\ge0.
\]
For \(\delta>0\), consider the Hermitian matrix valued function
\begin{equation}
 H_\delta(S):=\frac{cc^*}{(W+\delta)^2}.
 \label{eq:regularized-rank-one-direction}
\end{equation}
It is a Borel function of \(S\).  Each coordinate in a fixed real basis of
Hermitian matrices is bounded because
\[
 \frac{\abs{c_i}\abs{c_j}}{(W+\delta)^2}
 \le\frac{W}{(W+\delta)^2}\le\frac1{4\delta}.
\]
Write \(H_\delta(S)=\sum_\ell h_{\ell,\delta}(S)E_\ell\) in such a basis.
Apply \cref{prop:bounded-preserved-S-score} to each bounded scalar test
\(h_{\ell,\delta}\) and fixed Hermitian direction \(E_\ell\), then sum over
\(\ell\).  This gives
\begin{equation}
 \alpha_{k,m}\E\tr\!\left(M^{-1}\Delta(H_\delta(S))\right)
 =\E\tr\!\left(\Delta(H_\delta(S))\right).
 \label{eq:matrix-valued-preserved-score}
\end{equation}

The real and complex Gram coordinates are related by a fixed invertible block
change of basis.  The calculation in \cref{app:wishart-boundary} gives
\begin{align}
 \tr\!\left(\Delta(H_\delta(S))\right)
 &=\frac{W}{(W+\delta)^2},
 \label{eq:rank-one-score-trace}\\
 \tr\!\left(M^{-1}\Delta(H_\delta(S))\right)
 &=\frac{2\Xi}{(W+\delta)^2}.
 \label{eq:rank-one-inverse-score-trace}
\end{align}
Since \(2\alpha_{k,m}=k-2m-1\), substitution in
Eq.~\eqref{eq:matrix-valued-preserved-score} yields the exact regularized
identity
\begin{equation}
 \E\frac{\Xi}{(W+\delta)^2}
 =\frac1{k-2m-1}\E\frac{W}{(W+\delta)^2}.
 \label{eq:regularized-coupled-inverse-identity}
\end{equation}

Matrix Cauchy Schwarz and
\cref{prop:deterministic-schur-order} give, almost surely,
\begin{equation}
 W^2
 \le(c^*Qc)(c^*Q^{-1}c)
 \le V\Xi.
 \label{eq:cofactor-schur-cauchy}
\end{equation}
Thus \(V^{-1}\le\Xi/W^2\).  As \(\delta\) decreases to zero, the two
nonnegative integrands in Eq.~\eqref{eq:regularized-coupled-inverse-identity}
increase monotonically to \(\Xi/W^2\) and \(W^{-1}\).  Monotone convergence in
the extended nonnegative reals gives
\[
 \E V^{-1}
 \le\E\frac{\Xi}{W^2}
 =\frac1{k-2m-1}\E W^{-1},
\]
which proves Eq.~\eqref{eq:preserved-coordinate-inverse-variance}.
\end{proof}

\subsection{Hafnian cofactor specialization}

The odd hafnian cofactor vector is a polynomial, hence a Borel function, of
the preserved coordinate \(S=A_r^{\T}A_r\).  The general theorem therefore
gives the exact recursion used in the main proof.

\begin{corollary}[Self coupled hafnian direction]
\label[corollary]{cor:self-coupled-direction}
For \(r\ge2\) and \(k\ge4r\),
\begin{equation}
 \mathcal V_{k,r}
 \le \mathcal W_{k,r}\frac1{k-4r+1}
 \qquad\text{in }[0,\infty].
 \label{eq:V-to-W-inverse}
\end{equation}
\end{corollary}

\begin{proof}
Apply \cref{thm:preserved-coordinate-inverse-variance} with
\(m=2r-1\), \(A=A_r\), and \(c(S)=C_r\).  The nonvanishing condition was
proved in \cref{sec:preliminaries}, and
\[
 k-2m-1=k-4r+1.
\]
The definitions in Eq.~\eqref{eq:inverse-moment-definitions} now give
Eq.~\eqref{eq:V-to-W-inverse}.
\end{proof}

\begin{remark}[Why the loss is tolerable]
The denominator \(k-4r+1\) costs \(\exp(O(r/k))\) at level \(r\).
{At finite $n$ and $k$, its cumulative effect over $r\le n$ is
$\exp(O(n^2/k))$.  Along an asymptotic sequence $k=k_n$, this is polynomial
under \eqref{eq:poly-regime}.}
\end{remark}

%% file: sections/05_main_proof.tex
\section{Inverse moment bounds, coefficient, and exact local power}
\label{sec:main-proof}

\subsection{Proof of Theorem~\ref{thm:main-finite}}
\label{sec:proof-main-finite}

The two estimates obtained in \cref{sec:fourier,sec:wishart} alternate between
cofactor energy and conditional variance.  Their denominators telescope into
an explicit product, so the induction closes without a least singular value
event.

The bounded comparison used at each level can be recorded directly.  For
$r\ge2$, $k\ge4r$, and $\delta>0$, combining
Eqs.~\eqref{eq:regularized-coupled-inverse-identity} and
\eqref{eq:cofactor-schur-cauchy} gives
\begin{equation}
 \E\frac{W_r}{(W_r+\delta)^2}
 \ge (k-4r+1)\E\frac{W_r^2}
 {V_r(W_r+\delta)^2}.
 \label{eq:regularized-wishart-comparison}
\end{equation}
Thus this regularized inequality is a consequence of the preserved coordinate
identity and Schur order, rather than an additional input.

\begin{proposition}[Inverse variance]
\label[proposition]{prop:inverse-variance}
For $n\ge1$ and $k\ge4n$,
\begin{equation}
  \E V_n^{-1}
  \le\frac1{k-1}
  \prod_{r=2}^{n}
  \frac1{(2r-2)(k-4r+1)}.
  \label{eq:inverse-variance-product}
\end{equation}
\end{proposition}

\begin{proof}
At level one,
\[
  V_1=\norm{x_1}^2\sim\GammaLaw(k,1),
  \qquad
  \E V_1^{-1}=\frac1{k-1}.
\]
For $2\le r\le n$, combine
\cref{prop:fourier-to-laplace,cor:self-coupled-direction} to obtain
\[
  \E V_r^{-1}
  \le\frac1{(2r-2)(k-4r+1)}\E V_{r-1}^{-1}.
\]
The condition $k\ge4n$ makes every application valid.  Iterating
this recursion from the displayed base case proves
\eqref{eq:inverse-variance-product}.
\end{proof}

\begin{proof}[Proof of \cref{thm:main-finite}]
By \eqref{eq:disk-to-inverse-moment}, every disk of radius
$\varepsilon\sigma_{k,n}$ has probability at most
\[
  \varepsilon^2\sigma_{k,n}^2\E V_n^{-1}.
\]
Use \cref{prop:exact-second-moment,prop:inverse-variance} and the identity
\[
  \frac{(2n-1)!!}{\prod_{r=2}^{n}(2r-2)}
  =2n\,4^{-n}\binom{2n}{n}
\]
to obtain $\sigma_{k,n}^2\E V_n^{-1}\le B_{k,n}$, where the exact product
form is
\begin{equation}
 B_{k,n}=b_n\frac{k}{k-1}
 \prod_{r=2}^{n}\frac{k+2r-2}{k-4r+1}.
 \label{eq:B-product}
\end{equation}
This proves the bound by $B_{k,n}\varepsilon^2$ in
\eqref{eq:main-bound}; the additional cap by one follows because every event
has probability at most one.  The Gamma conversion, the full range bound, the sharper finite estimate,
and the two sided remainder are collected in
Appendix~\ref{app:second-moment}.  They give
Eqs.~\eqref{eq:B-exact} through \eqref{eq:B-bridge} and complete the proof.
\end{proof}

\subsection{Local density and exact local power}

Gaussian scale mixture density formulas are classical
\cite{AndrewsMallows1974}.  The next results isolate the exact analytic
hypotheses needed for this transpose-gram hafnian mixture.  The dimension condition
$k\ge4n$ is retained only where the explicit bound $B_{k,n}$ is invoked.

\begin{corollary}[Mixture density and exact local power]
\label[corollary]{cor:exact-local-power}
Let $n,k\ge1$.  Assume that $V_n>0$ almost surely and that
$\E V_n^{-1}<\infty$.  With respect to planar Lebesgue measure, the law of
$H_{k,n}$ has the bounded continuous density
\begin{equation}
  f_{k,n}(w)
  =\frac1\pi\E\!\left[
    V_n^{-1}\exp\!\left(-\frac{\abs{w}^2}{V_n}\right)
  \right].
  \label{eq:gram-hafnian-density}
\end{equation}
The density is strictly positive, radial, and nonincreasing in $\abs{w}$.
Writing $\Lambda_{k,n}:=\sigma_{k,n}^2\E V_n^{-1}$, for every $z\in\C$,
\begin{equation}
 \begin{aligned}
  0<\pi\sigma_{k,n}^2 f_{k,n}(z)&\le\Lambda_{k,n}<\infty,\\
  \Lambda_{k,n}&\le B_{k,n}\quad\text{if }k\ge4n.
 \end{aligned}
 \label{eq:local-coefficient-bound}
\end{equation}
Moreover,
\begin{equation}
 \begin{aligned}
  &\lim_{\varepsilon\downarrow0}
  \frac{\Pp\{\abs{H_{k,n}-z}
    \le\varepsilon\sigma_{k,n}\}}{\varepsilon^2}\\
  &\qquad=\pi\sigma_{k,n}^2 f_{k,n}(z)\\
  &\qquad=\sigma_{k,n}^2\E\!\left[
    V_n^{-1}e^{-\abs{z}^2/V_n}
  \right]>0.
 \end{aligned}
 \label{eq:exact-local-power}
\end{equation}
\end{corollary}

\begin{proof}
The conditional Gaussian law and $V_n>0$ almost surely give the mixture in
Eq.~\eqref{eq:gram-hafnian-density}.  The integrable majorant
$\pi^{-1}V_n^{-1}$ gives boundedness and continuity.  The same formula gives
strict positivity, radiality, monotonicity, and
$f_{k,n}(0)=\pi^{-1}\E V_n^{-1}$.  This proves
Eq.~\eqref{eq:local-coefficient-bound}, except for its last clause.  When
$k\ge4n$, that clause is the coefficient estimate in the proof of
\cref{thm:main-finite}.  Integration over a shrinking disk and continuity at
$z$ give Eq.~\eqref{eq:exact-local-power}.
\end{proof}

The radial factorization and Wick angular obstruction are proved in
Appendix~\ref{app:angular-obstruction}. Their density consequence retains
the explicit inverse-integrability hypothesis; it does not determine the
sharp threshold below $k\ge4n$.

\begin{remark}[What the induction avoids]
The pointwise inequality
\[
  V_r=\norm{A_rC_r}^2
  \ge s_{\min}(A_r)^2W_r
\]
is true, but it inserts the worst singular direction even though $C_r$ is a
highly structured function of $A_r^{\T}A_r$.  Iterating a worst direction
bound accumulates a substantially larger loss.  Conditioning on $S_r$
instead exploits the actual cofactor measurability and produces the explicit
$\exp(O(n^2/k))$ factor in \eqref{eq:B-simple}.
\end{remark}

%% file: sections/06_conditioning_consequences.tex
{
\section{Conditioning consequences for Gaussian GBS intensities}
\label{sec:conditioning-consequences}

The main theorem has a direct normalized intensity form.  For $z\in\C$ set
\begin{equation}
 Q_z:=\frac{\abs{H_{k,n}-z}^2}{\sigma_{k,n}^2}.
 \label{eq:normalized-shifted-intensity}
\end{equation}
The shift is mathematically useful for interpolation and perturbed analytic
targets.  Here $z\ne0$ is an analytic shift, not an optical displacement
parameter.  Only $z=0$ is used for the ordinary zero displacement GBS weight
discussed below.

\begin{corollary}[Normalized intensity tail]
\label[corollary]{cor:normalized-intensity-tail}
For $n\ge1$, $k\ge4n$, $z\in\C$, and $t\ge0$,
\begin{equation}
 \Pp\{Q_z\le t\}\le \min\{1,B_{k,n}t\}.
 \label{eq:normalized-intensity-tail}
\end{equation}
\end{corollary}

\begin{proof}
Apply Eq.~\eqref{eq:main-bound} with $\varepsilon=\sqrt t$.
\end{proof}

To connect this statement to GBS notation, let $N=2n$, $K=k$, $M\ge K$, and
consider equal squeezing $\xi>0$ in the Gaussian surrogate without collisions.  Its
pattern weight can be written
\begin{equation}
 \begin{aligned}
  p_G&=\alpha_{M,K,N}(\xi)\abs{H_{K,n}}^2=p_{\rm ref}Z,\\
  \alpha_{M,K,N}(\xi)&:=
  \frac{\tanh(\xi)^N}{M^N\cosh(\xi)^K},\\
  p_{\rm ref}&:=\alpha_{M,K,N}(\xi)\sigma_{K,n}^2,\\
  Z&:=Q_0.
 \end{aligned}
 \label{eq:gaussian-gbs-weight-normalization}
\end{equation}
This is the standard hafnian normalization for the Gaussian surrogate
pattern weight without collisions
\cite{HamiltonEtAl2017GBS,KruseEtAl2019GBS}.  The established density implies
$p_G>0$ almost surely, so the relative error ratio below is defined almost
surely.  \begin{corollary}[Gaussian additive to relative transfer]
\label[corollary]{cor:additive-to-relative-transfer}
Let $n\ge1$, $K\ge4n$, $N=2n$, $M\ge K$, $\xi>0$, and $\eta\ge0$, with
$p_G,p_{\rm ref}$ as in Eq.~\eqref{eq:gaussian-gbs-weight-normalization}.
Let $\widetilde p$ be a possibly randomized estimate of $p_G$ and suppose the
joint law of the Gaussian instance and $\widetilde p$ satisfies
$\Pp\{\abs{\widetilde p-p_G}>\eta p_{\rm ref}\}\le\delta_{\rm add}$.
Then for every $\rho>0$,
\begin{equation}
 \Pp\!\left\{\abs{\widetilde p-p_G}>\rho p_G\right\}
 \le \min\!\left\{1,\delta_{\rm add}+B_{K,n}\frac{\eta}{\rho}\right\}.
 \label{eq:additive-to-relative-transfer}
\end{equation}
\end{corollary}

\begin{proof}
Outside the additive failure event, relative error larger than $\rho$ implies
$Z<\eta/\rho$.  Apply Eq.~\eqref{eq:normalized-intensity-tail} at $z=0$
and take a union bound.
\end{proof}

Further consequences for fractional inverse moments and normalized
denominator resolution are given in Appendix~\ref{app:conditioning-details}.
These consequences isolate a conditioning step; they do not supply an
estimator or a complexity reduction.

}

%% file: sections/07_discussion.tex
\section{The unresolved small-ball transition}
\label{sec:polynomial-small-ball-transition}
Theorems~\ref{thm:main-finite} and
\ref{thm:symmetric-gaussian-anticoncentration} establish local
anticoncentration for the finite transpose Gram law and its independent
symmetric limit. Corollary~\ref{cor:historical-lower-tail} makes the
historical Gaussian lower-tail requirement quantitative, while
Corollary~\ref{cor:finite-haar-companion} transfers the finite theorem to
an ideal interferometer using the companion's hiding result.

The necessary and sufficient growth condition for local small-ball control
in the finite Gram ensemble is still open. For an integer sequence $k_n$, the question is whether
for every $A>0$ there exists $\gamma>0$ such that, eventually,
\begin{equation}
 \sup_{z\in\C}\Pp\{|H_{k_n,n}-z|\le n^{-\gamma}\sigma_{k_n,n}\}
       \le n^{-A}.
 \label{eq:open-transition-property}
\end{equation}
Corollary~\ref{cor:polynomial} proves this when
$n^2/k_n=O(\log n)$. It does not prove the necessity of that condition.
The restriction $k\ge4n$ is the range in which our inverse-moment
recursion closes, not a known statistical or computational threshold.

For the distinct moment criterion
$(\E|H_{k,n}|^2)^2/\E|H_{k,n}|^4\ge n^{-O(1)}$,
the $n^2/\log n$ transition is established
\cite{ZhaoExactMoments2026,ShouEtAl2026AnticoncentrationEntanglement}.
That result does not provide a necessity argument for the center-uniform
small-ball property in Eq.~\eqref{eq:open-transition-property}.

The radial and angular factorization in
Lemma~\ref{lem:radial-angular-factorization} isolates the small-$k$
obstruction. Theorem~\ref{thm:wick-angular-obstruction} proves an
exponential lower bound for the extended angular inverse-moment factor
when $n\ge1000$ and $2\le k\le n/3$. Its density conclusion retains
both positivity and inverse-integrability assumptions. An exponentially
large density coefficient at the origin does not, without a uniform
neighborhood estimate, settle Eq.~\eqref{eq:open-transition-property} at
inverse polynomial radii. Consequently there is no proved sharp transition
between that obstruction and the sufficient $n^2/\log n$ regime.

Another question is whether the preserved-coordinate inequality can be
sharpened for the hafnian cofactor direction. Its present denominator
$k-4r+1$ contributes an accumulated $\exp(O(n^2/k))$ loss. A sharper
self-coupled comparison could enlarge the polynomial small-ball regime
without appealing to a worst singular direction.

The scope of the perturbation theorems is also specific. The perturbation
is independent of the Gaussian matrix, and a scalar center is an analytic
shift, not an optical displacement. Adaptive shifts, displaced GBS and loop
hafnians, losses, and matrix discretization require separate results.
Finally, the average-case hafnian estimation conjectures and the
algorithmic parts of sampling reductions remain open inputs where they
are invoked \cite{HamiltonEtAl2017GBS,KruseEtAl2019GBS,DeshpandeEtAl2022,HangleiterEisert2023}.
The analytic lower-tail theorem removes an anticoncentration assumption
for the ensembles and regimes proved here; it does not remove those
computational assumptions.

%% file: sections/appendix_fourier_algebra.tex
\section{Fourier cofactor algebra and realification}
\label{app:fourier-algebra}

\begin{figure}[htbp]
  \centering
  \includegraphics[width=0.98\textwidth]{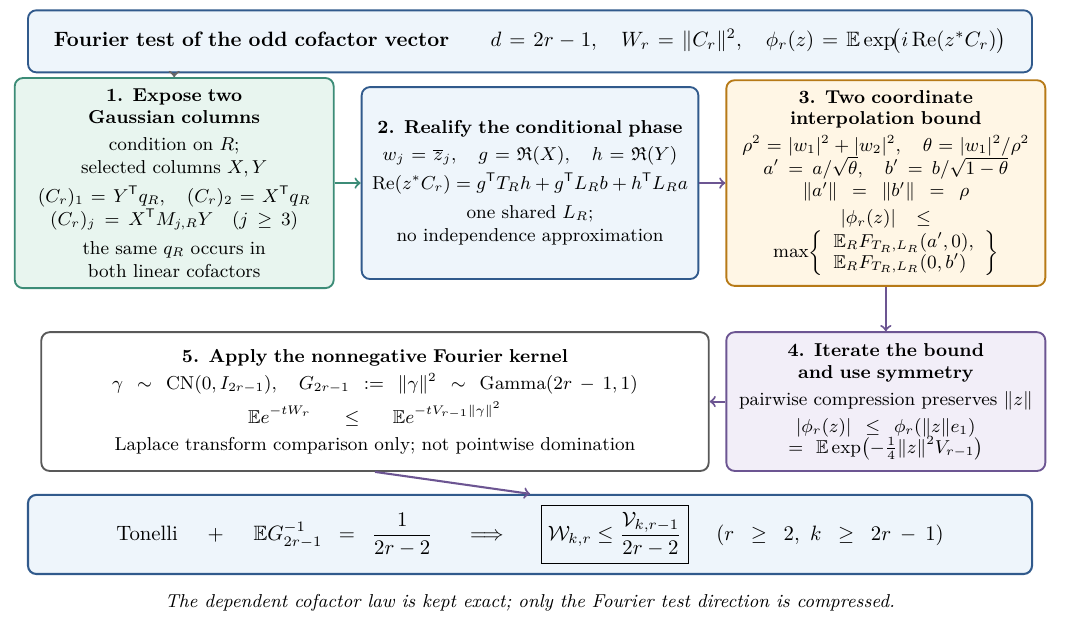}
  \caption{Fourier compression of odd hafnian cofactors.
  Conditioning on the remaining columns $R$ makes the first two cofactors
  share $q_R$, while the other cofactors are bilinear in the selected
  Gaussian columns $X$ and $Y$.  Realification therefore produces one
  Gaussian bilinear phase with the same map $L_R$ in both endpoint terms.
  The green arrow represents an exact identity.  Geometric interpolation and
  H\"older's inequality then bound two Fourier coefficients by an endpoint
  while preserving their squared Euclidean norm; the purple arrows represent
  inequalities, not distributional substitutions.  Iteration and symmetry
  give the radial characteristic function bound.  A nonnegative Gaussian
  Fourier kernel, Tonelli's theorem, and the Gamma inverse moment yield
  Eq.~\eqref{eq:W-inverse-recursion}.  The interpolation and positive kernel
  stages adapt Ref.~\cite{KoehlerLeung2026}; the shared map expansion is
  specific to the Gram hafnian.}
  \label{fig:fourier-cofactor-compression}
\end{figure}

The positive kernel compression and geometric interpolation strategy follows
Koehler and Leung
~\cite[Lemmas~3.1, 3.2, and~3.4 and Proposition~3.5]{KoehlerLeung2026}, while
complex realification is standard.  Compared with that permanent setting,
the new algebraic point is the Gram hafnian bilinear cofactor
expansion: it produces the same linear map in both selected coordinates and
therefore preserves the exact phase and scale required by the interpolation.

The next interpolation identity is the real Gaussian calculation in
Ref.~\cite[Lemma~3.2]{KoehlerLeung2026}, stated here with the shared map needed for our cofactors.

\begin{lemma}[Bilinear Gaussian interpolation]
\label[lemma]{lem:bilinear-interpolation}
Let $g,h$ be independent standard real Gaussian vectors in $\R^d$.  For real
matrices $T\in\R^{d\times d}$ and $L\in\R^{d\times2}$ and vectors
$a,b\in\R^2$, set
\begin{equation}
  F_{T,L}(a,b)
  :=\E\exp i\left(
    g^{\T}Th+g^{\T}Lb+h^{\T}La
  \right).
\end{equation}
Put $D=I+T^{\T}T$, choose an orthonormal eigenbasis
$(e_j)_{j=1}^d$ with
$De_j=\lambda_j e_j$, and set $q=T^{\T}Lb$.  Then every
$\lambda_j>0$ and the exact completed square formula is
\begin{align}
 F_{T,L}(a,b)
 &=\exp\!\left(-\frac{\norm{Lb}^2}{2}\right)
   \left(\prod_{j=1}^d\lambda_j^{-1/2}\right)\notag\\
 &\quad\times
   \exp\!\left\{
     \sum_{j=1}^d
     \frac{\bigl(-\ip{e_j}{q}+i\ip{e_j}{La}\bigr)^2}
          {2\lambda_j}
   \right\}.
\end{align}
In particular, $F_{T,L}(a,0)$ and $F_{T,L}(0,b)$ are strictly positive and,
for $0\le\theta\le1$,
\begin{equation}
  \abs{F_{T,L}(\sqrt\theta\,a,\sqrt{1-\theta}\,b)}
  =F_{T,L}(a,0)^\theta F_{T,L}(0,b)^{1-\theta}.
\end{equation}
\end{lemma}

\begin{proof}
Condition first on $h$.  Gaussian integration in $g$ replaces its phase by
$\exp(-\norm{Th+Lb}^2/2)$.  Including the standard Gaussian density of $h$
leaves
\begin{align}
 F_{T,L}(a,b)
 &=\exp\!\left(-\frac{\norm{Lb}^2}{2}\right)(2\pi)^{-d/2}\notag\\
 &\quad\times\int_{\R^d}
 e^{-h^{\T}Dh/2+\ip{h}{-q+iLa}}\,\dd h.
\end{align}
Write $h=\sum_j x_j e_j$.  The integral factors.  For $\lambda>0$ and
$s\in\C$, applying
\begin{equation}
 (2\pi)^{-1/2}\int_{\R}
 e^{-\lambda x^2/2+s x}\dd x
 =\lambda^{-1/2}e^{s^2/(2\lambda)}
\end{equation}
to each coordinate gives the stated spectral formula.

For the endpoint and interpolation claims, put
\begin{equation}
 P:=\prod_{j=1}^d\lambda_j^{-1/2},\qquad
 u_j:=\ip{e_j}{La},\qquad v_j:=\ip{e_j}{q}.
\end{equation}
Expanding $(-v_j+iu_j)^2$ shows that
\begin{align}
 &F_{T,L}(a,b)\notag\\
 &\quad=P\exp\!\left\{-\frac12\left(
   \norm{Lb}^2-\sum_j\frac{v_j^2}{\lambda_j}
   +\sum_j\frac{u_j^2}{\lambda_j}\right)\right\}\notag\\
 &\quad\times
 \exp\!\left(-i\sum_j\frac{u_jv_j}{\lambda_j}\right).
\end{align}
At either endpoint the phase vanishes, so $F_{T,L}(a,0)$ and
$F_{T,L}(0,b)$ are positive real numbers.  Under
$(a,b)\mapsto(\sqrt\theta\,a,\sqrt{1-\theta}\,b)$, the two real quadratic
terms scale by $\theta$ and $1-\theta$, respectively, whereas their common
positive prefactor is $P$.  Taking absolute values proves the geometric
identity.
\end{proof}

\begin{proof}[Proof of \cref{prop:cofactor-compression}]
We first compress two nonzero coordinates.  Put $d=2r-1$, relabel the
corresponding columns of $A_r$ as $X,Y\in\C^k$, and write the remaining
columns as $R=(r_a)_{a\in I}$, where $I=\{3,\ldots,d\}$.  For
$J\subseteq I$ of even cardinality, let
\begin{equation}
 h(J):=\haf(R_J^{\T}R_J),
 \qquad h(\varnothing)=1.
\end{equation}
Define
\begin{equation}
 q_R:=\sum_{a\in I}h(I\setminus\{a\})r_a
\end{equation}
and, for $j\in I$,
\begin{equation}
 M_{j,R}:=h(I\setminus\{j\})I_k
 +\sum_{\substack{a,b\in I\setminus\{j\}\\a\ne b}}
 h(I\setminus\{j,a,b\})r_ar_b^{\T}.
\end{equation}
Direct hafnian expansion then gives
\begin{equation}
  \begin{aligned}
  (C_r)_1&=Y^{\T}q_R,
  & (C_r)_2&=X^{\T}q_R,\\
  (C_r)_j&=X^{\T}M_{j,R}Y
  &&(j\ge3).
  \end{aligned}
\end{equation}
For the first identity, expand the even hafnian defining $(C_r)_1$ along its
unique displayed column $Y$; the coefficient of $Y^{\T}r_a$ is precisely
$h(I\setminus\{a\})$.  The second identity is identical.  For $j\ge3$, the
matchings that pair $X$ to $Y$ contribute
$h(I\setminus\{j\})X^{\T}Y$.  If $X$ is paired to $r_a$ and $Y$ to $r_b$,
with $a\ne b$, the contribution is
\begin{equation}
 h(I\setminus\{j,a,b\})(X^{\T}r_a)(Y^{\T}r_b).
\end{equation}
Summing these alternatives gives the displayed formula for $M_{j,R}$ and the
third cofactor identity.

Put $w_j=\overline{z_j}$.  Conditional on $R$, realify
complex vectors by
\begin{equation}
 \mathfrak R(x):=\sqrt2
 \begin{pmatrix}\operatorname{Re}x\\ \operatorname{Im}x\end{pmatrix}
 \in\R^{2k}.
\end{equation}
Thus $g=\mathfrak R(X)$ and $h=\mathfrak R(Y)$ are independent standard real
Gaussian vectors.  If $q=q_0+iq_1$, set
\begin{equation}
 L(q):=\frac1{\sqrt2}
 \begin{pmatrix}q_0&-q_1\\-q_1&-q_0\end{pmatrix}
 \in\R^{2k\times2}.
\end{equation}
For $w=u+iv$ and $M=M_0+iM_1$, set
\begin{equation}
 \mathcal T(w,M):=\frac12
 \begin{pmatrix}
  uM_0-vM_1&-uM_1-vM_0\\
  -uM_1-vM_0&-uM_0+vM_1
 \end{pmatrix}.
\end{equation}
A direct expansion of real and imaginary parts gives the exact identities
\begin{align}
 \operatorname{Re}(w x^{\T}q)
 &=\mathfrak R(x)^{\T}L(q)\binom uv,
 \\
 \operatorname{Re}(w x^{\T}My)
 &=\mathfrak R(x)^{\T}\mathcal T(w,M)\mathfrak R(y).
\end{align}
Now take
\begin{equation}
 \begin{aligned}
 a&=\binom{\operatorname{Re}w_1}{\operatorname{Im}w_1},
 &\qquad
 b&=\binom{\operatorname{Re}w_2}{\operatorname{Im}w_2},\\
 L_R&=L(q_R),
 &
 T_R&=\sum_{j=3}^{d}\mathcal T(w_j,M_{j,R}).
 \end{aligned}
\end{equation}
Then $\norm a=\abs{w_1}$, $\norm b=\abs{w_2}$, and
the preceding cofactor and realification formulas show that the conditional
real phase $\operatorname{Re}(z^*C_r)$ is exactly
\begin{equation}
  g^{\T}T_Rh+g^{\T}L_Rb+h^{\T}L_Ra.
\end{equation}
In particular, the same matrix $L_R$ occurs in both linear terms, with no
hidden rescaling, because both endpoint cofactors use the same vector $q_R$.
This is the structural point that permits
\cref{lem:bilinear-interpolation} to apply while preserving the Euclidean norm
of the two Fourier coefficients.

Let $\rho^2=\abs{w_1}^2+\abs{w_2}^2$ and suppose both terms are nonzero.  Set
$\theta=\abs{w_1}^2/\rho^2$, $a'=a/\sqrt\theta$, and
$b'=b/\sqrt{1-\theta}$.  Thus $\norm{a'}=\norm{b'}=\rho$.  By
\cref{lem:bilinear-interpolation}, the triangle inequality, and H\"older's
inequality,
\begin{align}
  \abs{\phi_r(z)}
  &\le \E_R\abs{F_{T_R,L_R}(a,b)}\notag\\
  &\le
  \bigl(\E_R F_{T_R,L_R}(a',0)\bigr)^\theta
  \bigl(\E_R F_{T_R,L_R}(0,b')\bigr)^{1-\theta}\notag\\
  &\le \max\left\{
    \E_R F_{T_R,L_R}(a',0),
    \E_R F_{T_R,L_R}(0,b')
  \right\}.
\end{align}
The two quantities on the last line are characteristic functions obtained by
deleting one selected coordinate and replacing the other by a coefficient of
modulus $\rho$.  They are nonnegative because their conditional endpoint
values are strictly positive.  If one of the original coordinates vanishes,
the same conclusion is immediate.

Iterating this two coordinate estimate leaves one nonzero coordinate
and preserves $\norm z$.  The coordinates of $C_r$ are exchangeable.
Moreover, multiplying every column of $A_r$ by $e^{i\vartheta}$ multiplies
$C_r$ by $e^{i(2r-2)\vartheta}$.  Since the law of $A_r$ is invariant under
the former multiplication and the power map on the unit circle is onto, the
law of $C_r$ is invariant under every common phase.  The surviving
coordinate may therefore be placed first and its coefficient made real and
nonnegative.  This proves
\begin{equation}
  \abs{\phi_r(z)}\le\phi_r(\norm z e_1).
\end{equation}

Finally, $(C_r)_1$ has the same law as $H_{k,r-1}$.  The fixed past
pushforward law in \cref{prop:conditional-mixture} gives its singleton
characteristic function as
\begin{equation}
  \phi_r(t e_1)=\E e^{-\abs t^2V_{r-1}/4},
\end{equation}
which proves the last equality in \eqref{eq:cofactor-compression}.
\end{proof}

%% file: sections/appendix_wishart_boundary.tex
\section{Unsliced Wishart score and cutoff proof}
\label{app:wishart-boundary}

Haff's identity supplies the classical Wishart integration by parts
framework~\cite{Haff1979}.  Here we derive the fixed transpose-gram
specialization used in \cref{cor:self-coupled-direction} directly in Gaussian
matrix coordinates.  The singular lifting field, cutoff limit, and subsequent
bounded Borel extension are proved here rather than quoted as
a theorem of Haff.  This route avoids differentiating a conditional
Wishart density on a moving domain.

Write $A=X+iY\in\C^{k\times m}$ and
$R=[X\;Y]\in\R^{k\times 2m}$.  The entries of $R$ are independent real
Gaussians with density proportional to $e^{-x^2}$.  Set
\[
  \begin{aligned}
    M&=R^{\T}R, & Q&=A^*A,\\
    S&=A^{\T}A, &
    K&=\begin{pmatrix}Q&\overline S\\S&\overline Q\end{pmatrix}.
  \end{aligned}
\]

\begin{lemma}[Unsliced preserved $S$ score]
\label[lemma]{lem:wishart-boundary-regularity}
Suppose $k>2m+1$ and put $\alpha=(k-2m-1)/2$.  Let $H=H_0+iH_1$ be
Hermitian, with $H_0^{\T}=H_0$ and $H_1^{\T}=-H_1$.  For every continuously
differentiable compactly supported real test function $\psi$ on the real
vector space of complex symmetric matrices, both terms below are integrable
and
\[
  \E\!\left[\psi(S)\left(
    \alpha\tr(M^{-1}\Delta(H))-\tr(\Delta(H))
  \right)\right]=0,
\]
where
\[
 \Delta(H)=\frac12
 \begin{pmatrix}H_0&H_1\\-H_1&H_0\end{pmatrix}.
\]
\end{lemma}

\begin{proof}
Put $D=\Delta(H)$.  The realification identities give
\[
  \begin{aligned}
    \partial_DQ&=H, & \partial_DS&=0,\\
    \tr D&=\tr H.
  \end{aligned}
\]

Work first on the full rank set.  The matrix field
\[
  \mathcal F_D(R)=\frac12 R\,(R^{\T}R)^{-1}D
\]
lifts the direction $D$: its first variation under the Gram map
$R\mapsto R^{\T}R$ is exactly $D$.  Direct differentiation in the $2mk$
real matrix coordinates gives
\begin{equation}
 \begin{aligned}
 R^{\T}\mathcal F_D(R)+\mathcal F_D(R)^{\T}R&=D,\\
 \operatorname{div}\mathcal F_D(R)
 &=\frac{k-2m-1}{2}\tr(M^{-1}D),\\
 2\ip{R}{\mathcal F_D(R)}_{\mathrm F}&=\tr D.
 \end{aligned}
 \label{eq:lifted-field-identities}
\end{equation}
The first identity and $\partial_DS=0$ show that the derivative of
$\psi(S(R))$ along this field vanishes.

We next justify using this singular field.  For $\eta>0$, set
\[
 \begin{aligned}
 \mathcal F_{D,\eta}(R)&:=\frac12R(M+\eta I)^{-1}D,\\
 P_\eta&:=M(M+\eta I)^{-1}.
 \end{aligned}
\]
Choose a smooth radial cutoff $\chi_L$ that equals one on
$\{\norm{R}_{\mathrm F}\le L\}$ and vanishes outside
$\{\norm{R}_{\mathrm F}\le2L\}$.  We use the following Gaussian divergence
identity in finite dimensions.  If $u\in\R^d$, if
$\gamma_{1/2}^{\otimes d}$ is the product $N(0,1/2)$ law, and if the scalar
components $F_j$ and their coordinate derivatives satisfy the usual slice
differentiability and integrability hypotheses, then
\begin{equation}
 \begin{aligned}
 &\int_{\R^d}\sum_{j=1}^d \partial_jF_j(u)\,
   \dd\gamma_{1/2}^{\otimes d}(u)\\
 &\quad=\int_{\R^d}\sum_{j=1}^d 2u_jF_j(u)\,
   \dd\gamma_{1/2}^{\otimes d}(u).
 \end{aligned}
 \label{eq:cutoff-gaussian-ibp}
\end{equation}
This identity is field independent.  Vectorize the $2mk$
real entries of $R$ into $u$ and take $F_j$ to be the components of the
compactly supported field
$\chi_L(R)\psi(S(R))\mathcal F_{D,\eta}(R)$.  Its smooth cutoff supplies the
slice and integrability hypotheses, and Eq.~\eqref{eq:cutoff-gaussian-ibp}
then specializes to
\[
 \E\operatorname{div}\!\left(
   \chi_L\psi(S)\mathcal F_{D,\eta}\right)
 =2\E\!\left[
   \chi_L\psi(S)\ip{R}{\mathcal F_{D,\eta}}_{\mathrm F}
 \right].
\]
For fixed $\eta$, the cutoff derivative term tends to zero as $L$ tends to
infinity by Gaussian integrability.  The first variation of $M$ along the
regularized field is
\[
 \dot M_\eta=\frac12(P_\eta D+DP_\eta).
\]
On the almost sure full rank set, $P_\eta\to I$.  Hence
$\dot M_\eta\to D$, and the derivative of $\psi(S)$ tends to zero because
$\partial_DS=0$.

{Put $B_\eta=(M+\eta I)^{-1}$.  Direct differentiation in the
real matrix coordinates gives
\[
 \operatorname{div}\mathcal F_{D,\eta}
 =\frac12\Bigl[(k-\tr P_\eta)\tr(B_\eta D)
      -\tr(B_\eta M B_\eta D)\Bigr].
\]
On the full rank set this converges to
$\alpha\tr(M^{-1}D)$.  Because $0\preceq P_\eta\preceq I$ and
$0\preceq B_\eta M B_\eta\preceq B_\eta\preceq M^{-1}$,
\[
 \abs{\operatorname{div}\mathcal F_{D,\eta}}
 \le \frac{k+1}{2}\norm{D}_{\mathrm{op}}\tr(M^{-1}),
 \qquad 0<\eta\le1.
\]
Moreover, $\norm{\dot M_\eta}_{\mathrm F}\le\norm D_{\mathrm F}$, so the
derivative of the compactly supported test is bounded by a constant depending
only on the test and $D$.  Also,
$2\ip{R}{\mathcal F_{D,\eta}}_{\mathrm F}\to\tr D$, and this term is bounded
by $2m\norm D_{\mathrm{op}}$.  Thus the full divergence integrand is dominated
by $C_{m,D,\psi}(1+\tr(M^{-1}))$.}
The standard inverse Wishart mean is
\[
 \begin{aligned}
 \E M^{-1}&=\frac{2}{k-2m-1}I_{2m},\\
 \E\tr(M^{-1})&=\frac{4m}{k-2m-1},
 \end{aligned}
\]
which is finite in the stated range.  This is the standard inverse Wishart
mean, a special case of the general inverse moment formulas
~\cite{Muirhead1982,Matsumoto2012}.  Dominated convergence therefore permits
first $L\to\infty$ and then $\eta\downarrow0$ in
\eqref{eq:cutoff-gaussian-ibp}.  Using
\eqref{eq:lifted-field-identities} gives
\[
 \alpha\E\!\left[\psi(S)\tr(M^{-1}D)\right]
 =\E\!\left[\psi(S)\tr D\right],
\]
as claimed.  The inverse trace bound also proves the two integrability
statements.
\end{proof}

We record the block calculation used for the matrix valued test of rank one in
the proof of \cref{cor:self-coupled-direction}.  Let
\[
 T=\begin{pmatrix}I&I\\ iI&-iI\end{pmatrix}.
\]
Then
\[
 [A\;\overline A]=RT,
 \qquad K=T^*MT,
 \qquad T^*T=TT^*=2I,
\]
and direct multiplication gives
\[
 \Delta(H)=\frac14T
 \begin{pmatrix}H&0\\0&\overline H\end{pmatrix}T^*.
\]
Cyclicity of trace consequently yields
\begin{equation}
 \tr(M^{-1}\Delta(H))
 =\tr\!\left(
 K^{-1}\begin{pmatrix}H&0\\0&\overline H\end{pmatrix}
 \right).
 \label{eq:real-complex-score-trace}
\end{equation}
If $H=cc^*/(w+\delta)^2$, the lower right block of $K^{-1}$ is the
entrywise conjugate of its upper left block.  Since $K^{-1}$ is Hermitian,
the two diagonal block contributions in
\eqref{eq:real-complex-score-trace} are equal and real.  The block inverse
formula and \eqref{eq:deterministic-schur-order} therefore give
\begin{equation}
 \tr(M^{-1}\Delta(H))
 =\frac{2c^*\mathcal S(A)^{-1}c}{(w+\delta)^2}.
 \label{eq:rank-one-real-complex-score}
\end{equation}
Also,
\begin{equation}
 \tr(\Delta(H))=\tr(H)=\frac{c^*c}{(w+\delta)^2}.
 \label{eq:rank-one-direction-trace}
\end{equation}
Taking $c=C_r$, $w=W_r$, and $m=2r-1$ gives precisely
\eqref{eq:rank-one-score-trace} and
\eqref{eq:rank-one-inverse-score-trace}.

Smooth compactly supported tests determine finite signed measures on this
$S$ space in finite dimensions \cite[Chapters~7 and~8]{Folland1999RealAnalysis}.
The bounded Borel extension used in
\eqref{eq:V-to-W-inverse} therefore follows from the same identity.

\subsection{Bounded Borel extension}

The preceding cutoff argument produces the smooth compactly supported score
identity as a specialization of the classical framework.  The following
finite signed measure argument permits the test direction to
depend measurably on the preserved transpose-gram coordinate.

\begin{proof}[Proof of \cref{prop:bounded-preserved-S-score}]
The smooth test identity is proved directly in the Gaussian coordinates in
\cref{lem:wishart-boundary-regularity}.  We explain here why it gives the
stated bounded Borel result.  Since \(M\) is a real Wishart matrix of size
\(2m\),
\begin{equation}
 \E\tr(M^{-1})=\frac{4m}{k-2m-1}<\infty.
 \label{eq:real-wishart-inverse-trace}
\end{equation}
Consequently,
\[
 \abs{\tr(M^{-1}\Delta(H))}
 \le \norm{\Delta(H)}_{\mathrm{op}}\tr(M^{-1})
\]
is integrable.  Define a finite signed measure on the real vector space of
complex symmetric matrices by
\begin{align}
 \nu_H(E):={}&\int_{\{R:S(R)\in E\}}
 \bigl[\alpha_{k,m}\tr(M(R)^{-1}\Delta(H))\notag\\
 &\hspace{7.5em}-\tr(\Delta(H))\bigr]\,\dd\mu_{k,m}(R).
 \label{eq:score-signed-measure}
\end{align}
The identity in \cref{lem:wishart-boundary-regularity} says that the integral
of every continuously differentiable compactly supported test against
\(\nu_H\) vanishes.  Such tests determine finite signed measures on this
real space in finite dimensions \cite[Chapters~7 and~8]{Folland1999RealAnalysis}, so
\(\nu_H=0\).  Integrating a bounded Borel
function \(u\) against this zero measure gives
\eqref{eq:bounded-preserved-S-score}.  The same inverse trace estimate and the
boundedness of \(u\) give the two asserted integrability statements.
\end{proof}

%% file: sections/appendix_second_moment.tex
\section{Exact coefficient and coefficient asymptotics}
\label{app:second-moment}

\subsection{Completion of the proof of Theorem~\ref{thm:main-finite}}
\label{sec:proof-main-coefficient}

The product to Gamma conversion and the logarithmic expansions are elementary
once the two inverse moment transitions have been proved.  The coefficient
bounds and critical regime expansions below follow from those
transitions, rather than from prior moment formulas.  They
are recorded here to keep the main proof focused on the alternating
recursion.

Put $\mathcal R^{\mathrm{AC}}_{k,n}:=B_{k,n}/b_n$ for the normalized finite coefficient.

\input{sections/01_coefficient_regimes_visuals}

\begin{proof}[Completion of the proof of \cref{thm:main-finite}]
The product in Eq.~\eqref{eq:B-exact} has the equivalent Gamma form
\begin{equation}
 B_{k,n}=b_n\frac{k}{k-1}\frac{2^{n-1}}{4^{n-1}}
 \frac{\Gamma(k/2+n)}{\Gamma(k/2+1)}
 \frac{\Gamma((k-4n+1)/4)}{\Gamma((k-3)/4)}.
 \label{eq:B-gamma}
\end{equation}
Applying the Gamma recurrence \cite[Eq.~5.5.1]{DLMFGamma} separately to
the numerator and denominator products gives
\begin{align*}
 \prod_{r=2}^{n}(k+2r-2)
 &=2^{n-1}\frac{\Gamma(k/2+n)}{\Gamma(k/2+1)},\\
 \prod_{r=2}^{n}(k-4r+1)
 &=4^{n-1}\frac{\Gamma((k-3)/4)}
 {\Gamma((k-4n+1)/4)},
\end{align*}
which proves Eq.~\eqref{eq:B-gamma}.

Two elementary branches give a uniform estimate on the full theorem range.
For $n\ge1$ and $k\ge8n$,
\begin{equation}
 \log\prod_{r=2}^{n}\frac{k+2r-2}{k-4r+1}
 \le\frac{6n^2-6}{k}.
 \label{eq:B-log-branch}
\end{equation}
For $n\ge1$ and $4n\le k<8n$, monotonicity in $k$ and reversal of the
denominator product give
\begin{equation}
 \frac{k}{k-1}\prod_{r=2}^{n}\frac{k+2r-2}{k-4r+1}
 \le\frac{2(6n)^{n-1}}{(n-1)!}\le e^{4n}.
 \label{eq:B-factorial-branch}
\end{equation}
For $k\ge8n$, also $k/(k-1)<2$; for $4n\le k<8n$, the head ratio is already
included in Eq.~\eqref{eq:B-factorial-branch}.  Cauchy--Schwarz in Euler's Gamma integral
\cite[Eq.~5.2.1]{DLMFGamma} gives
$\Gamma(n+1/2)^2\le\Gamma(n)\Gamma(n+1)=n\Gamma(n)^2$.
Thus
$b_n\le2\sqrt{n/\pi}\le2\sqrt n$.  Using the latter convenient form, the two
branches yield the full range bound
\begin{equation}
 B_{k,n}\le4\sqrt n\exp\!\left(\frac{32n^2}{k}\right),
 \qquad n\ge1,\quad k\ge4n.
 \label{eq:B-uniform}
\end{equation}

For the sharper finite estimate, put
\[
 a_r:=6r-3,
 \qquad
 \frac{k+2r-2}{k-4r+1}=1+\frac{a_r}{k-4r+1}.
\]
The identity
\[
 \frac{a_r}{k-4r+1}
 =\frac{a_r}{k}
 +\frac{a_r(4r-1)}{k(k-4r+1)}
\]
and the finite sums
\begin{align*}
 \sum_{r=2}^{n}a_r&=3n^2-3,\\
 \sum_{r=2}^{n}a_r(4r-1)&=8n^3+3n^2-2n-9
\end{align*}
make the constants explicit.  Using
$\log(1+x)\le x$ and
\[
 \log\frac{k}{k-1}
 \le\frac1{k-1}=\frac1k+\frac1{k(k-1)}
\]
then gives
\[
 \log\frac{B_{k,n}}{b_n}
 \le \frac{3n^2-2}{k}
 +\frac{9n^3}{k(k-4n+1)},
\]
which is Eq.~\eqref{eq:B-simple}.

For the two sided remainder, use
$x-x^2/2\le\log(1+x)\le x$ together with
\[
 \frac12\sum_{r=2}^{n}a_r^2
 =6n^3-\frac32n-\frac92\le19n^3.
\]
The same decomposition gives the exact representation
\[
 \begin{gathered}
 \theta_{k,n}:=\log\frac{B_{k,n}}{b_n}-\frac{3n^2}{k},\\
 \begin{aligned}
 \abs{\theta_{k,n}}
 &\le \frac2k+\frac{9n^3}{k(k-4n+1)}\\
 &\quad+\frac{19n^3}{(k-4n+1)^2}.
 \end{aligned}
 \end{gathered}
\]
For $k\ge8n$ the right hand side is at most
$2/k+94n^3/k^2$, proving Eq.~\eqref{eq:B-bridge}.  Finally, Stirling's
Gamma ratio asymptotic \cite[Eq.~5.11.12]{DLMFGamma} gives $b_n/(2\sqrt{n/\pi})\to1$.
\end{proof}

\subsection{Derivation of the coefficient regime table}

The six asymptotics in Table~\ref{tab:coefficient-regimes} follow from the
exact product; we record the details
because they are theoretical claims even though they are not numbered
equations.  At the two linear endpoints, Eq.~\eqref{eq:B-exact} gives
\begin{align*}
 \mathcal R^{\mathrm{AC}}_{4n,n}
 &=\frac{4n}{4n-1}\,2^{-(n-1)}
   \frac{\Gamma(3n)}{\Gamma(2n+1)}
   \frac{\Gamma(1/4)}{\Gamma(n-3/4)},\\
 \mathcal R^{\mathrm{AC}}_{8n,n}
 &=\frac{8n}{8n-1}\,2^{-(n-1)}
   \frac{\Gamma(5n)}{\Gamma(4n+1)}
   \frac{\Gamma(n+1/4)}{\Gamma(2n-3/4)}.
\end{align*}
Stirling's formula, applied to each displayed Gamma ratio, yields
\begin{align*}
 \mathcal R^{\mathrm{AC}}_{4n,n}&\sim\frac{\Gamma(1/4)}{\sqrt{3\pi}}\,n^{1/4}
   \left(\frac{27}{8}\right)^n,\\
 \mathcal R^{\mathrm{AC}}_{8n,n}&\sim\frac{2^{5/4}}{\sqrt5}
   \left(\frac{3125}{2048}\right)^n.
\end{align*}

For $k/n\to\infty$, expand the logarithm of the exact product in the form
\[
 \log \mathcal R^{\mathrm{AC}}_{k,n}=\log\frac{k}{k-1}
 +\sum_{r=2}^{n}\log\left(1+\frac{6r-3}{k-4r+1}\right).
\]
Using the finite sums already displayed in the proof above and expanding to
second order gives, uniformly whenever $k/n\to\infty$,
\[
 \log \mathcal R^{\mathrm{AC}}_{k,n}=\frac{3n^2}{k}+\frac{2n^3}{k^2}
 +O\left(\frac1k+\frac{n^2}{k^2}+\frac{n^4}{k^3}\right).
\]
Substitution gives, respectively,
\begin{align*}
 \log \mathcal R^{\mathrm{AC}}_{\lceil n^{3/2}\rceil,n}
   &=3\sqrt n+2+O(n^{-1/2}),\\
 \log \mathcal R^{\mathrm{AC}}_{\lceil n^2/\log n\rceil,n}
   &=3\log n+\frac{2(\log n)^2}{n}\\
   &\quad+O\left(\frac{(\log n)^3}{n^2}\right),\\
 \log \mathcal R^{\mathrm{AC}}_{n^2,n}&=3+\frac2n+O(n^{-2}),\\
 \log \mathcal R^{\mathrm{AC}}_{n^3,n}&=\frac3n+O(n^{-3}).
\end{align*}
Exponentiating proves the four corresponding rows of
Table~\ref{tab:coefficient-regimes}.  Changing a noninteger $k$ by at most
one changes these expansions only within the displayed remainders.

\subsection{Polynomial regime}
\label{sec:polynomial-regime}

\begin{proof}[Proof of \cref{cor:polynomial}]
Under \eqref{eq:poly-regime},
\[
  \frac{k_n}{n}\ge\frac{n}{D\log n}\longrightarrow\infty.
\]
Thus $k_n\ge8n$ for all sufficiently large $n$.  The remainder in
Eq.~\eqref{eq:B-bridge} tends to zero, because
$n^3/k_n^2\le D^2(\log n)^2/n\to0$.  Hence, for all sufficiently large $n$,
\[
  B_{k_n,n}
  \le2e\sqrt n\exp\!\left(\frac{3n^2}{k_n}\right)
  \le2e\,n^{3D+1/2}.
\]
At radius $n^{-\gamma}$, \cref{thm:main-finite} therefore gives, for every
$z\in\C$,
\[
  \Pp\{\abs{H_{k_n,n}-z}\le n^{-\gamma}\sigma_{k_n,n}\}
  \le2e\,n^{3D+1/2-2\gamma}.
\]
If \eqref{eq:radius-exponent-choice} holds with a fixed strict margin, the last
expression is at most $n^{-A}$ for all sufficiently large $n$.
\end{proof}

%% file: sections/01_coefficient_regimes_visuals.tex
\begin{figure}[tbp]
  \centering
  \begingroup
  \setlength{\unitlength}{1pt}
  \resizebox{0.96\textwidth}{!}{%
  \begin{picture}(504,238)
    \put(0,0){\includegraphics[width=504pt]{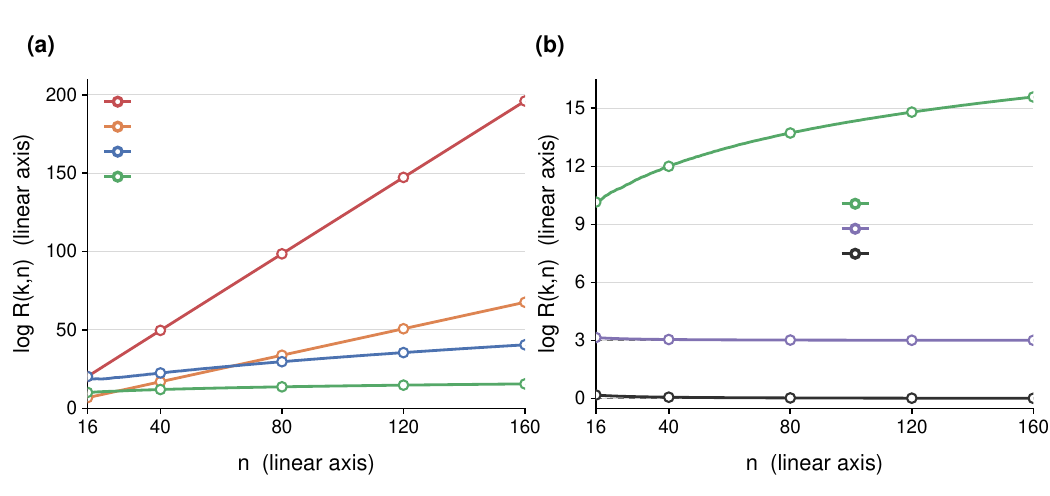}}
    \put(67,187){\makebox(0,0)[l]{\fontsize{9}{10}\selectfont $k=4n$}}
    \put(67,175){\makebox(0,0)[l]{\fontsize{9}{10}\selectfont $k=8n$}}
    \put(67,163){\makebox(0,0)[l]{\fontsize{9}{10}\selectfont $k=n^{3/2}$}}
    \put(67,151){\makebox(0,0)[l]{\fontsize{9}{10}\selectfont $k=n^2/\log n$}}
    \put(421,138){\makebox(0,0)[l]{\fontsize{9}{10}\selectfont $k=n^2/\log n$}}
    \put(421,126){\makebox(0,0)[l]{\fontsize{9}{10}\selectfont $k=n^2$}}
    \put(421,114){\makebox(0,0)[l]{\fontsize{9}{10}\selectfont $k=n^3$}}
  \end{picture}}
  \endgroup
  \caption{Natural logarithm of the exact normalized coefficient
  $\mathcal R^{\mathrm{AC}}_{k,n}$, evaluated from the exact product equivalent to
  Eq.~\eqref{eq:B-exact} for
  $16\le n\le160$.  The plot labels $R$ denote $\mathcal R^{\mathrm{AC}}$. The horizontal $n$ axis and the vertical axis are both
  linear; the ordinate is the natural logarithm $\log \mathcal R^{\mathrm{AC}}_{k,n}$.  Exact values
  at integer $n$ are joined by line segments, and markers identify selected
  checkpoints.  When a displayed relation does not give an integer $k$, the
  plotted value is rounded upward, although the ceiling sign is omitted from
  the legend.  Panel (a) shows the rapid penalty in the linear and logarithmic
  boundary regimes, including the subquadratic relation $k=n^{3/2}$.
  Panel (b) compares the logarithmic boundary with the quadratic and cubic
  scales: $\log \mathcal R^{\mathrm{AC}}_{\lceil n^2/\log n\rceil,n}\sim3\log n$,
  $\log \mathcal R^{\mathrm{AC}}_{n^2,n}\to3$, and $\log \mathcal R^{\mathrm{AC}}_{n^3,n}\to0$.  Dashed lines mark the
  limits $3$ and $0$.}
  \label{fig:coefficient-regimes}
\end{figure}

\begin{table}[tbp]
  \centering
  \small
  \renewcommand{\arraystretch}{1.15}
  \caption{Coefficient regimes obtained from the exact product equivalent to
  Eq.~\eqref{eq:B-exact}.  The linear rows use Stirling asymptotics, while the
  remaining rows follow by expanding that product.  Upward integer rounding
  in the two nonintegral sequences changes none of the displayed asymptotics.}
  \label{tab:coefficient-regimes}
  \begin{tabular}{lcc}
  \toprule
  Dimension relation & Normalized coefficient $\mathcal R^{\mathrm{AC}}_{k_n,n}$ & Interpretation \\
  \midrule
  $k_n=4n$
    & $\displaystyle\sim\frac{\Gamma(1/4)}{\sqrt{3\pi}}n^{1/4}
      \left(\frac{27}{8}\right)^n$
    & Edge of the finite theorem \\
  \addlinespace[3pt]
  $k_n=8n$
    & $\displaystyle\sim\frac{2^{5/4}}{\sqrt5}
      \left(\frac{3125}{2048}\right)^n$
    & Linear exponential penalty \\
  \addlinespace[3pt]
  $k_n=\lceil n^{3/2}\rceil$
    & $e^2\exp(3\sqrt n)(1+O(n^{-1/2}))$
    & Subexponential, but superpolynomial \\
  \addlinespace[3pt]
  $k_n=\lceil n^2/\log n\rceil$
    & $n^3(1+O((\log n)^2/n))$ & Polynomial boundary \\
  \addlinespace[3pt]
  $k_n=n^2$ & $e^3(1+2/n+O(n^{-2}))$
    & Constant penalty in finite dimensions \\
  \addlinespace[3pt]
  $k_n=n^3$ & $1+3/n+O(n^{-2})$
    & Limiting coefficient recovered \\
  \bottomrule
  \end{tabular}
\end{table}

%% file: sections/appendix_symmetric_limit.tex
\section{Symmetric Gaussian limit and perturbation proofs}
\label{app:direct-symmetric}

This appendix proves
\cref{cor:independent-matrix-shift,thm:symmetric-gaussian-anticoncentration,cor:symmetric-gaussian-perturbation}.
The symmetric Gaussian theorem is the fixed degree limit of
\cref{thm:main-finite}.  The two perturbation corollaries retain their direct
Gaussian comparison proofs.

\subsection{Fixed degree limit proof of the symmetric Gaussian theorem}

\begin{proof}[Proof of \cref{thm:symmetric-gaussian-anticoncentration}]
Fix $n$.  Write $X^{(k)}=(x_{ai})_{1\le a\le k,\,1\le i\le2n}$.
For one row, collect the upper triangular products
$(x_{ai}x_{aj})_{1\le i\le j\le2n}$ into a finite dimensional complex
vector.  These row vectors are independent, identically distributed, and
centered.  Wick's rule gives complex variance one for every off diagonal
coordinate, complex variance two for every diagonal coordinate, zero
covariance between distinct upper triangular coordinates, and zero
pseudocovariance.  The finite dimensional multivariate central limit theorem
therefore gives
\begin{equation}
 k^{-1/2}(X^{(k)})^{\T}X^{(k)}
 \ \Longrightarrow\ G_{2n}^{\mathrm{sym}}
 \qquad(k\longrightarrow\infty).
 \label{eq:symmetric-gram-clt}
\end{equation}
Here $\Longrightarrow$ denotes weak convergence of probability laws,
equivalently convergence in distribution, in the finite dimensional real
coordinate space formed by the real and imaginary parts of the upper
triangular entries.  Here and below $n$ is fixed, and $k$ tends to infinity
through the positive integers.  The central limit theorem and the weak
convergence results used below are standard; see Ref.~\cite{Kallenberg2021}.

The hafnian is a continuous polynomial and satisfies
$\haf(cA)=c^n\haf(A)$.  The continuous mapping theorem applied to
Eq.~\eqref{eq:symmetric-gram-clt} yields
\begin{equation}
 k^{-n/2}H_{k,n}
 =\haf\!\left(k^{-1/2}(X^{(k)})^{\T}X^{(k)}\right)
 \ \Longrightarrow\ H_n^{\mathrm{sym}}.
 \label{eq:symmetric-hafnian-weak-limit}
\end{equation}
The exact second moment in Eq.~\eqref{eq:sigma-main} gives
\begin{equation}
 \begin{aligned}
 \frac{k^{-n}\sigma_{k,n}^2}{(2n-1)!!}
 &=\prod_{q=0}^{n-1}\left(1+\frac{2q}{k}\right),\\
 k^{-n}\sigma_{k,n}^2
 &\longrightarrow (2n-1)!!=(\sigma_n^{\mathrm{sym}})^2,\\
 k^{-n/2}\sigma_{k,n}&\longrightarrow\sigma_n^{\mathrm{sym}}.
 \end{aligned}
 \label{eq:symmetric-scale-limit}
\end{equation}
In addition, the fixed degree coefficient limit in
\cref{thm:main-finite} is
\begin{equation}
 B_{k,n}\longrightarrow b_n.
 \label{eq:symmetric-coefficient-limit}
\end{equation}

Put $U_{k,n}:=H_{k,n}/\sigma_{k,n}$ and
$U_n^{\mathrm{sym}}:=H_n^{\mathrm{sym}}/\sigma_n^{\mathrm{sym}}$.
Equations~\eqref{eq:symmetric-hafnian-weak-limit} and
\eqref{eq:symmetric-scale-limit}, together with Slutsky's theorem, give
$U_{k,n}\Longrightarrow U_n^{\mathrm{sym}}$.  Theorem~\ref{thm:main-finite}
applied at the unnormalized center $a\sigma_{k,n}$, where
$a:=z/\sigma_n^{\mathrm{sym}}$, bounds the corresponding normalized disk
probability for every $k\ge4n$.  For any $\eta>0$, the open set part of the Portmanteau
theorem and Eq.~\eqref{eq:symmetric-coefficient-limit} give
\begin{equation}
 \begin{aligned}
 \Pp\!\left\{\abs{U_n^{\mathrm{sym}}-a}\le\varepsilon\right\}
 &\le
 \Pp\!\left\{\abs{U_n^{\mathrm{sym}}-a}<\varepsilon+\eta\right\}\\
 &\le \liminf_{k\to\infty}
 \Pp\!\left\{\abs{U_{k,n}-a}<\varepsilon+\eta\right\}\\
 &\le \min\{1,b_n(\varepsilon+\eta)^2\}.
 \end{aligned}
 \label{eq:symmetric-portmanteau-chain}
\end{equation}
Letting $\eta\downarrow0$ proves
Eq.~\eqref{eq:symmetric-gaussian-small-ball}, including
$\varepsilon=0$.

\end{proof}

\subsection{Perturbations of the rectangular Gaussian factor}

The proof follows the Gaussian perturbation comparison for Ginibre
permanents of Koehler and Leung
\cite[Theorem~4.1]{KoehlerLeung2026}.  Here the same conditioning and
successive Gaussian replacement argument is applied to columns, using
separate linearity of the transpose-gram hafnian polynomial in place of
permanent multilinearity.

\begin{proof}[Proof of \cref{cor:independent-matrix-shift}]
For $Y=(y_1,\ldots,y_{2n})$, set
$F(y_1,\ldots,y_{2n})=\haf(Y^{\T}Y)$.  Every perfect matching monomial
contains each column label exactly once.  Consequently, $F$ is complex linear
in every column separately.

Let $g,h\sim\CN(0,I_k)$. For fixed $a\in\C^k$, $b\in\C$
and $\rho\ge0$, the circular Gaussian disk-centering inequality gives
\begin{equation}
 \Pp_g\{|g^{\T}a+b|\le\rho\}
 \le \Pp_h\{|h^{\T}a|\le\rho\}.
 \label{eq:scalar-gaussian-shift-disk}
\end{equation}
Indeed, for $a\ne0$ the linear form has circular Gaussian variance
$\|a\|^2$, and a centered disk maximizes the integral of its radial,
decreasing density. This is the two-dimensional Gaussian case of
Anderson's inequality \cite{Anderson1955}. If $a=0$, the right side is
one, so the inequality still holds, including $\rho=0$.

Fix the complete perturbation $\Delta=\delta$. Its columns need not be
independent. Conditional on every Gaussian column except the first,
separate linearity gives
\[
 F(g_1+\delta_1,g_2+\delta_2,\ldots)-w
 =g_1^{\T}a+\delta_1^{\T}a-w.
\]
Apply Eq.~\eqref{eq:scalar-gaussian-shift-disk} and integrate the other
columns. On the right, the first shift and the external center have both
been removed, while the first Gaussian column retains its original law.
For each subsequent column the same conditioning removes its shift.
Fubini's theorem applies to these bounded Borel indicators. Thus
\begin{equation}
 \Pp_X\{|F(X+\delta)-w|\le\rho\}
 \le \Pp_X\{|F(X)|\le\rho\}.
 \label{eq:independent-factor-shift-disk-chain}
\end{equation}
Independence of $X$ and $\Delta$ identifies their joint law with a product;
integrating this pointwise-in-$\delta$ inequality therefore gives the
same comparison for arbitrary random $\Delta$, with no moment condition.
Taking $\rho=\varepsilon\sigma_{k,n}$ and applying
\cref{thm:main-finite} at center zero proves
Eq.~\eqref{eq:independent-matrix-shift-bound}, including $\varepsilon=0$.
The argument compares disk probabilities directly, so no exponential
indicator majorant and no extra factor $e$ are needed.
\end{proof}

\subsection{Direct proof for additive symmetric matrix perturbations}

\begin{proof}[Proof of \cref{cor:symmetric-gaussian-perturbation}]
Condition first on $\Delta$.  Its diagonal entries may be discarded because
they do not occur in any hafnian monomial.  For a fixed vertex $v$, condition
on all Gaussian edges outside the complete star at $v$.  Every perfect
matching uses exactly one edge in that star.

More precisely, the hafnian is complex linear in the whole Gaussian star.
Conditionally, the star contribution has the form $g^{\T}c+b$, where $g$ is a
standard circular complex Gaussian vector.  Thus
$g^{\T}c\sim\CN(0,\norm{c}^2)$, and radial monotonicity of its density gives
\begin{equation}
 \Pp\{\abs{g^{\T}c+b}\le\rho\}
 \le \Pp\{\abs{g^{\T}c}\le\rho\}.
 \label{eq:symmetric-star-scalar-centering}
\end{equation}
After each replacement the edge family remains independent Gaussian, with
zero shift on every previously visited edge, so overlapping stars cause no
difficulty.

The final comparison is with the unperturbed centered symmetric Gaussian
hafnian.

Averaging this deterministic shift comparison over the independent law of
$\Delta$ and applying \cref{thm:symmetric-gaussian-anticoncentration} at
center zero proves Eq.~\eqref{eq:symmetric-gaussian-perturbation-bound} with
the same coefficient $b_n$.  No boundedness, moment, or entrywise independence
assumption on $\Delta$ is used.  Independence from the Gaussian matrix is
essential: the choice $\Delta=-G_{2n}^{\mathrm{sym}}$ would collapse the
shifted hafnian to zero.

\end{proof}

%% file: sections/appendix_angular_obstruction.tex
\section{Radial factorization and the Wick angular obstruction}
\label{app:angular-obstruction}

The radial factorization below assumes positivity and, where stated,
inverse integrability. \Cref{prop:inverse-variance} supplies these hypotheses
when $k\ge4n$. The Wick obstruction is an extended-integral statement
and does not assume that sufficient condition.

\begin{lemma}[Inverse integrability and radial and angular factorization]
\label[lemma]{lem:radial-angular-factorization}
Let $n\ge1$ and $k\ge2$.  Assume that $V_n>0$ almost surely, and put
$d=2n-1$.  For the columns $x_j$ of $A_n$, write
$x_j=\sqrt{R_j}\,u_j$, where the $R_j$ are independent with law
$\GammaLaw(k,1)$ and the $u_j$ are independent uniform unit vectors in
$\C^k$, independent of the radii.  Put
$\widehat A_n=[u_1,\ldots,u_d]$, let $\widehat C_n$ be its cofactor vector,
and define the directional cofactor energy by
$\widehat V_{k,n}:=\norm{\widehat A_n\widehat C_n}^2$.  Then
\begin{equation}
  V_n=\left(\prod_{j=1}^{d}R_j\right)\widehat V_{k,n}.
  \label{eq:radial-polar-decomposition}
\end{equation}
Moreover, $\widehat V_{k,n}>0$ almost surely, and
\[
 V_n^{-1}\text{ is integrable}
 \quad\Longleftrightarrow\quad
 \widehat V_{k,n}^{-1}\text{ is integrable}.
\]
Whenever these equivalent conditions hold,
\[
  \E V_n^{-1}
  =(k-1)^{-d}\E\widehat V_{k,n}^{-1}.
\]
Consequently,
\begin{equation}
  \Lambda_{k,n}
  =\underbrace{\left(\frac{k}{k-1}\right)^d}_{R_{k,n}^{\mathrm{rad}}}
   \underbrace{\bigl(\E\widehat V_{k,n}\bigr)
   \bigl(\E\widehat V_{k,n}^{-1}\bigr)}_{\mathcal A_{k,n}},
  \qquad \mathcal A_{k,n}\ge1.
  \label{eq:radial-angular-factorization}
\end{equation}
In particular,
\begin{equation}
  \Lambda_{k,n}
  \ge \left(\frac{k}{k-1}\right)^{2n-1}
  \ge \exp\!\left(\frac{2n-1}{k}\right)
  \ge \exp\!\left(\frac{n}{k}\right).
  \label{eq:radial-lower-bound}
\end{equation}
\end{lemma}

\begin{proof}
Put $\widehat A_n=[u_1,\ldots,u_d]$ and let $\widehat C_n$ be its cofactor
vector.  For an even symmetric matrix $B$ and a diagonal matrix
$D_0=\diag(s_1,\ldots,s_{2m})$, the perfect matching definition gives
\[
  \haf(D_0BD_0)=\left(\prod_{\ell=1}^{2m}s_\ell\right)\haf(B),
\]
because every perfect matching uses each index exactly once.  Applying this
identity to the Gram matrix after deleting column $j$ gives
\[
 \begin{aligned}
  (C_n)_j
  &=\left(\prod_{\ell\ne j}\sqrt{R_\ell}\right)
    (\widehat C_n)_j,\\
  A_nC_n
  &=\left(\prod_{\ell=1}^{d}\sqrt{R_\ell}\right)
    \widehat A_n\widehat C_n.
 \end{aligned}
\]
Taking squared Hermitian norms proves
Eq.~\eqref{eq:radial-polar-decomposition}.

Every radius is positive, so the assumed positivity of $V_n$ and
Eq.~\eqref{eq:radial-polar-decomposition} imply that
$\widehat V_{k,n}>0$ almost surely.  The reciprocal radial product is
integrable because $k\ge2$ and
$\E R_j^{-1}=(k-1)^{-1}$.  Independence and Tonelli's theorem therefore show
that $V_n^{-1}$ is integrable exactly when
$\widehat V_{k,n}^{-1}$ is integrable.  Under these equivalent conditions,
Fubini's theorem and the identities $\E R_j=k$ and
$\E R_j^{-1}=(k-1)^{-1}$ give
\[
  \E V_n=k^d\E\widehat V_{k,n},\qquad
  \E V_n^{-1}=(k-1)^{-d}\E\widehat V_{k,n}^{-1}.
\]
Since $\E V_n=\sigma_{k,n}^2$, multiplying the two identities proves
Eq.~\eqref{eq:radial-angular-factorization}.  Almost sure positivity and
the Cauchy Schwarz inequality give
\[
  1=\left(\E\!\left[
  \widehat V_{k,n}^{1/2}\widehat V_{k,n}^{-1/2}\right]\right)^2
  \le(\E\widehat V_{k,n})(\E\widehat V_{k,n}^{-1}).
\]
Finally, $\log(k/(k-1))\ge1/k$ and $2n-1\ge n$ prove
Eq.~\eqref{eq:radial-lower-bound}.
\end{proof}

Thus \cref{prop:inverse-variance} proves that, for any integer sequence
$(k_n)$ satisfying $k_n\ge4n$ eventually, the model with $k=k_n$ satisfies
$\E V_n^{-1}<\infty$ eventually.  This is a sufficient condition, not
a known numerical characterization.  For $k\ge2$, the lemma reduces the
exact finiteness question to the angular inverse moment; its threshold below
the proved range $k\ge4n$ remains open.

\begin{theorem}[Wick angular obstruction]
\label{thm:wick-angular-obstruction}
Let $n,k\ge1$ and put $d=2n-1$.  Let $u_1,\ldots,u_d$ be independent
uniform unit vectors in $\C^k$, put
$\widehat A_n=[u_1,\ldots,u_d]$, and let $\widehat C_n$ be its cofactor
vector.  Define
$\widehat V_{k,n}:=\norm{\widehat A_n\widehat C_n}^2$, and interpret its
inverse moment as an extended nonnegative integral, with
$0^{-1}=+\infty$.  Also define
\[
 m_k:=\frac{\Gamma(3/2)\Gamma(k)}{\Gamma(k+1/2)},\qquad
 D_{k,n}:=\prod_{q=0}^{n-1}(k+2q),
\]
and
\begin{equation}
 \begin{aligned}
 \mathcal Q^{\mathrm{Wick}}_{k,n}
 &:=\frac{(2n-1)!!}{D_{k,n}}
 \left(\frac{1}{km_k^2}\right)^{2n-1}\\
 &=\prod_{q=0}^{n-1}\frac{2q+1}{k+2q}
 \left(\frac{4k\binom{2k}{k}^{2}}{16^k}\right)^{2n-1}.
 \end{aligned}
 \label{eq:wick-angular-ratio}
\end{equation}
Then the extended angular factor satisfies
\begin{equation}
 \begin{aligned}
 \mathcal I^{\mathrm{ext}}_{k,n}
 &:=\int \widehat V_{k,n}^{-1}\,\dd\Pp\in[0,\infty],\\
 \mathcal A^{\mathrm{ext}}_{k,n}
 &:=(\E\widehat V_{k,n})\mathcal I^{\mathrm{ext}}_{k,n}
 \ge\max\{1,\mathcal Q^{\mathrm{Wick}}_{k,n}\}.
 \end{aligned}
 \label{eq:wick-angular-obstruction}
\end{equation}
In particular, if $n\ge1000$, $2\le k$, and $3k\le n$, then
\begin{equation}
 \mathcal A^{\mathrm{ext}}_{k,n}\ge\exp(n/500).
 \label{eq:wick-angular-finite-range}
\end{equation}
If, in addition, $V_n>0$ almost surely and $\E V_n^{-1}<\infty$, then the
continuous density in \cref{cor:exact-local-power} exists and the same finite
range gives
\[
 \begin{aligned}
  \pi\sigma_{k,n}^2 f_{k,n}(0)
  &=\Lambda_{k,n}\ge\exp(n/500),\\
  f_{k,n}(0)&\ge\frac{\exp(n/500)}{\pi\sigma_{k,n}^2}.
 \end{aligned}
\]
This density conclusion is conditional: the present inverse moment argument
does not establish $\E V_n^{-1}<\infty$ in the range $3k\le n$.
\end{theorem}

\begin{proof}
For the directions in the theorem and a standard real Gaussian vector
$g\in\R^k$, the Wick identity, applied coordinatewise and extended complex
linearly, \cite{Isserlis1918} gives
\begin{equation}
 \widehat A_n\widehat C_n
 =\E_g\!\left[g\prod_{j=1}^{d}(u_j^{\T}g)\right].
\label{eq:wick-vector-representation}
\end{equation}

Indeed, if $\widehat A_{n,-j}$ denotes the matrix with column $j$ deleted,
then every coordinate $1\le a\le k$ satisfies
\begin{equation}
 \begin{aligned}
 \E_g\!\left[g_a\prod_{j=1}^{d}(u_j^{\T}g)\right]
 &=\sum_{j=1}^{d}\E_g[g_a u_j^{\T}g]
   \haf\!\left(\widehat A_{n,-j}^{\T}\widehat A_{n,-j}\right)\\
 &=\sum_{j=1}^{d}(u_j)_a(\widehat C_n)_j
  =(\widehat A_n\widehat C_n)_a.
 \end{aligned}
 \label{eq:wick-coordinate-representation}
\end{equation}
This coordinate identity is the Wick pairing behind
Eq.~\eqref{eq:wick-vector-representation}.

Minkowski's inequality, Tonelli's theorem, independence of the directions,
and unitary invariance therefore give the explicit intermediate estimate
\begin{equation}
 \begin{aligned}
 \E_u\sqrt{\widehat V_{k,n}}
 &=\E_u\left\|\E_g\!\left[g\prod_{j=1}^{d}(u_j^{\T}g)\right]\right\|\\
 &\le\E_g\!\left[\norm{g}
       \prod_{j=1}^{d}\E_{u_j}\abs{u_j^{\T}g}\right]\\
 &=m_k^d\E_g\norm{g}^{d+1}
 =m_k^dD_{k,n}.
 \end{aligned}
\label{eq:wick-minkowski-chain}
\end{equation}
Here $\E|u^{\T}g|=m_k\norm{g}$, and the last equality is the exact real
Gaussian moment.  The exact first angular moment, obtained from the Gaussian
polar decomposition and \cref{prop:exact-second-moment}, is
\begin{equation}
 \E\widehat V_{k,n}=\frac{(2n-1)!!D_{k,n}}{k^d}.
\label{eq:wick-angular-first-moment}
\end{equation}
If $\mathcal I^{\mathrm{ext}}_{k,n}=+\infty$, the desired inequality is
immediate because the displayed first angular moment is positive.  Suppose
instead that $\mathcal I^{\mathrm{ext}}_{k,n}<+\infty$.  Then
$\widehat V_{k,n}>0$ almost surely.  Put
$V:=\widehat V_{k,n}$ and
$\mathcal I:=\mathcal I^{\mathrm{ext}}_{k,n}$.  Apply Cauchy Schwarz first
to the one quarter and negative one quarter powers of $V$, and then to
$V^{-1/2}$ and the constant one.  Squaring the resulting inequalities gives
the first line below.  The ordinary reciprocal moment inequality gives the
second:
\begin{equation}
 \begin{aligned}
 1&\le(\E\sqrt V)^2\mathcal I,\\
 1&\le(\E V)\mathcal I=\mathcal A^{\mathrm{ext}}_{k,n}.
 \end{aligned}
\label{eq:wick-cauchy-chain}
\end{equation}
Equations~\eqref{eq:wick-minkowski-chain},
\eqref{eq:wick-angular-first-moment}, and the first inequality in
Eq.~\eqref{eq:wick-cauchy-chain} give
$\mathcal A^{\mathrm{ext}}_{k,n}\ge
\mathcal Q^{\mathrm{Wick}}_{k,n}$.  The second inequality in
Eq.~\eqref{eq:wick-cauchy-chain} gives
$\mathcal A^{\mathrm{ext}}_{k,n}\ge1$.  Together these prove
Eq.~\eqref{eq:wick-angular-obstruction}.  The half integer Gamma identity

\begin{equation}
 m_k=\frac{4^k}{2k\binom{2k}{k}}
\label{eq:wick-half-integer-gamma}
\end{equation}
gives the exact rational form in
Eq.~\eqref{eq:wick-angular-ratio}.

For completeness, we establish the remaining finite range using two exact
rational recurrences.  Writing
$a_k=4k\binom{2k}{k}^2/16^k$, the central binomial recurrence and
$n\ge3k$ give the analytic first step
\begin{equation}
 \begin{aligned}
 a_{k+1}
 &=a_k\frac{(2k+1)^2}{4k(k+1)},\\
 a_k&\ge\frac98,
 &\frac{2n+1}{k+2n}&\ge\frac67,\\
 \frac{\mathcal Q^{\mathrm{Wick}}_{k,n+1}}
 {\mathcal Q^{\mathrm{Wick}}_{k,n}}
 &=\frac{2n+1}{k+2n}a_k^2
 \ge\frac{243}{224}>\frac{500}{499}.
 \end{aligned}
\label{eq:wick-ratio-recurrence}
\end{equation}
Thus this recurrence proves the first step directly.
For $k\ge334$, the boundary values obey
\begin{equation}
 \mathcal Q^{\mathrm{Wick}}_{k+2,3(k+2)}
 \ge\left(\frac{500}{499}\right)^6
 \mathcal Q^{\mathrm{Wick}}_{k,3k}.
\label{eq:wick-boundary-recurrence}
\end{equation}
To verify this uniform step explicitly, the Gamma recurrence in
Eq.~\eqref{eq:wick-angular-ratio} gives
\begin{equation}
\begin{aligned}
\frac{\mathcal Q^{\mathrm{Wick}}_{k+2,3k+6}}
 {\mathcal Q^{\mathrm{Wick}}_{k,3k}}
={}&\frac{(k/2)\prod_{j=0}^{5}(3k+1/2+j)}
 {\prod_{j=0}^{6}(7k/2+j)}\\
&\quad\times a_k^{12}
 \left(\frac{a_{k+2}}{a_k}\right)^{6k+11}.
\end{aligned}
\label{eq:wick-boundary-ratio-explicit}
\end{equation}
The six paired factors are at least $6/7$, $a_{k+2}\ge a_k$, and
$k/(7k+12)\ge334/2350$ for $k\ge334$. Therefore the right side is at least
\[
 \frac{334}{2350}\left(\frac67\right)^6a_{334}^{12}
 >\left(\frac{500}{499}\right)^6.
\]
This last comparison is a single exact rational inequality with
$a_{334}=1336\binom{668}{334}^2/16^{334}$; it is included in the
arithmetic certificate accompanying the sources.

It remains to prove the following three base inequalities:
\begin{equation}
 \begin{aligned}
 \left(\frac{500}{499}\right)^{1000}
 &\le\mathcal Q^{\mathrm{Wick}}_{k,1000}
 &&(2\le k\le333),\\
 \left(\frac{500}{499}\right)^{1002}
 &\le\mathcal Q^{\mathrm{Wick}}_{334,1002},\\
 \left(\frac{500}{499}\right)^{1005}
 &\le\mathcal Q^{\mathrm{Wick}}_{335,1005}.
 \end{aligned}
 \label{eq:wick-finite-base-cases}
\end{equation}
Here is a reproducible finite verification rule. Define the positive
integers
\[
\begin{aligned}
 U(k,n)&=\prod_{q=0}^{n-1}(2q+1)
       \left(4k\binom{2k}{k}^2\right)^{2n-1},\\
 V(k,n)&=\prod_{q=0}^{n-1}(k+2q)\,16^{k(2n-1)}.
\end{aligned}
\]
For each $(k,n)$ in $\{(k,1000):2\le k\le333\}$ together with
$(334,1002)$ and $(335,1005)$, compare the integers
$499^nU(k,n)$ and $500^nV(k,n)$. In every one of these 334 comparisons
the first is at least the second. The source certificate implements precisely
these products with integer arithmetic; the Lean finite certificate checks
the same inequalities in the kernel. No floating-point approximation is
used. The last two cases initiate the even and odd $k\mapsto k+2$
inductions. The $n\mapsto n+1$ recurrence then covers every larger $n$.

Since
$\exp(1/500)\le500/499$, induction yields
\begin{equation}
 \mathcal Q^{\mathrm{Wick}}_{k,n}
 \ge\left(\frac{500}{499}\right)^n
 \ge\exp(n/500).
\label{eq:wick-exponential-chain}
\end{equation}
This proves
Eq.~\eqref{eq:wick-angular-finite-range}.  The final assertion follows from
Eq.~\eqref{eq:radial-angular-factorization}, $R_{k,n}^{\mathrm{rad}}\ge1$, and
Eq.~\eqref{eq:gram-hafnian-density} at the origin.
\end{proof}

%% file: sections/appendix_conditioning_details.tex
\section{Further conditioning consequences}
\label{app:conditioning-details}

The lower tail also controls how strongly rare small weights amplify division
errors on a fractional average.  This is more informative than a single
failure threshold, but it does not assert a finite ordinary inverse moment.

\begin{corollary}[Negative fractional moments]
\label[corollary]{cor:fractional-negative-moments}
Let $n\ge1$, $k\ge4n$, and $z\in\C$.  For every $0<s<1$,
\begin{equation}
 \E\,Q_z^{-s}\le \frac{B_{k,n}^{\,s}}{1-s}.
 \label{eq:fractional-negative-moment}
\end{equation}
More generally, if a positive random variable $Y$ satisfies
$\Pp\{Y\le t\}\le\min\{1,Bt^\alpha\}$ for all $t\ge0$, where
$B,\alpha>0$, then for $0<s<\alpha$,
\begin{equation}
 \E\,Y^{-s}\le
 \frac{\alpha}{\alpha-s}B^{s/\alpha}.
 \label{eq:general-negative-moment}
\end{equation}
For the present Gram hafnian variable, the exact local density theorem also
gives the complementary statement
\begin{equation}
 \E Q_z^{-s}=+\infty,\qquad s\ge1.
 \label{eq:critical-negative-moment-threshold}
\end{equation}
\end{corollary}

\begin{proof}
The layer cake representation for nonnegative random variables
\cite[Chap.~1]{Kallenberg2021} gives
\[
 \E Y^{-s}=s\int_0^\infty
 t^{-s-1}\Pp\{Y\le t\}\,\dd t.
\]
Split the integral at $t=B^{-1/\alpha}$ and use respectively $Bt^\alpha$
and $1$.  This proves Eq.~\eqref{eq:general-negative-moment};
Eq.~\eqref{eq:fractional-negative-moment} follows from
\cref{cor:normalized-intensity-tail} with $\alpha=1$.
For the converse range, Eq.~\eqref{eq:exact-local-power} implies
\[
 \lim_{t\downarrow0}\frac{\Pp\{Q_z\le t\}}{t}
 =\pi\sigma_{k,n}^2 f_{k,n}(z)>0.
\]
Thus $\Pp\{Q_z\le t\}\ge ct$ for some $c,t_0>0$ and every
$0<t<t_0$.  The same layer cake formula bounds $\E Q_z^{-s}$ below by
$sc\int_0^{t_0}t^{-s}\,\dd t$, which diverges exactly when $s\ge1$.
\end{proof}

Use the Gaussian weight normalization in Eq.~\eqref{eq:gaussian-gbs-weight-normalization}. If $\eta\ge0$ and an additive estimate obeys
$\abs{\widetilde p-p_G}\le\eta p_{\rm ref}$ almost surely, then
\begin{equation}
 \E\!\left[\left(
 \frac{\abs{\widetilde p-p_G}}{p_G}\right)^s\right]
 \le \frac{(\eta B_{K,n})^s}{1-s},
 \qquad 0<s<1.
 \label{eq:relative-error-fractional-moment}
\end{equation}
Indeed, the integrand is at most $\eta^sZ^{-s}$.  For an estimator whose
additive guarantee holds only on a good event, this moment conclusion applies
to the good event contribution; controlling the contribution on failure
requires an additional bound on the estimator.

The exponent range is exact: the linear lower tail estimate proves finiteness
below order one, while the already established strictly positive local density
proves divergence at and above order one.  Thus the negative moment
consequence both quantifies average numerical conditioning and identifies the
critical threshold.

\begin{corollary}[Normalized denominator resolution]
\label[corollary]{cor:normalized-denominator-resolution}
Let $n\ge1$, $k\ge4n$, and $z\in\C$.  For every $b\ge0$,
\begin{equation}
 \Pp\{Q_z\le2^{-b}\}\le\min\{1,B_{k,n}2^{-b}\}.
 \label{eq:normalized-denominator-tail}
\end{equation}
Consequently, for $0<\delta<1$, the choice
\begin{equation}
 b\ge \log_2\!\left(\frac{B_{k,n}}{\delta}\right)
 \label{eq:normalized-denominator-bits}
\end{equation}
makes the small normalized denominator probability at most $\delta$.
\end{corollary}

\begin{proof}
Set $t=2^{-b}$ in Eq.~\eqref{eq:normalized-intensity-tail}.
\end{proof}

The denominator resolution consequence counts additional \emph{normalized
denominator} bits.  It does not
control total output probability precision, experimental sample complexity,
or perturbation stability of a hafnian with respect to its matrix entries.
Likewise, Eqs.~\eqref{eq:relative-error-fractional-moment} and
\eqref{eq:additive-to-relative-transfer} isolate one conditioning step that
appears in approximate counting arguments
\cite{AaronsonArkhipov2013}; they do not establish an estimator, its running
time, a Stockmeyer reduction for this ensemble, or hardness for typical instances.
Interpolation, normalization, and rejection procedures can all create inverse
quantities, so the negative moment bound is a reusable input, not a complete
complexity reduction.